\documentclass{amsart}

\usepackage{amscd,amssymb,amsmath,graphicx}
\usepackage{mathrsfs}
\newcommand\mylabel[1]{\label{#1}}

\newtheorem{theorem}{Theorem}[section]
\newtheorem*{maintheorem}{Theorem}
\newtheorem{lemma}[theorem]{Lemma}
\newtheorem{proposition}[theorem]{Proposition}
\newtheorem{corollary}[theorem]{Corollary}

\theoremstyle{definition}
\newtheorem{definition}[theorem]{Definition}
\newtheorem{example}[theorem]{Example}

\newtheorem{remark}[theorem]{Remark}
\newtheorem*{acknowledgement}{Acknowledgement}

\theoremstyle{remark}

\DeclareFontFamily{U}{wncy}{}
\DeclareFontShape{U}{wncy}{m}{n}{<->wncyr10}{}
\DeclareSymbolFont{mcy}{U}{wncy}{m}{n}
\DeclareMathSymbol{\Sh}{\mathord}{mcy}{"58}
\newcommand{\ZZ}	{\mathbb{Z}}
\newcommand{\QQ}	{\mathbb{Q}}
\newcommand{\RR}	{\mathbb{R}}
\newcommand{\CC}	{\mathbb{C}}

\newcommand{\PP}	{\mathbb{P}}

\newcommand{\GG}	{\mathbb{G}}

\newcommand{\ideala}    {\mathfrak{a}}
\newcommand{\idealb}    {\mathfrak{b}}
\newcommand{\idealm}    {\mathfrak{m}}
\newcommand{\maxid }    {\mathfrak{m}}

\newcommand  {\shA}     {\mathscr{A}}

\newcommand  {\shE}     {\mathscr{E}}
\newcommand  {\shF}     {\mathscr{F}}

\newcommand  {\shI}     {\mathscr{I}}
\newcommand  {\shJ}     {\mathscr{J}}

\newcommand  {\shM}     {\mathscr{M}}

\newcommand  {\shN}     {\mathscr{N}}
\newcommand  {\shL}     {\mathscr{L}}

\newcommand  {\shY}     {\mathscr{Y}}

\newcommand  {\foX}     {\mathfrak{X}}

\newcommand  {\an}      {{\text{an}}}

\newcommand  {\Ass}     {\operatorname{Ass}}
\newcommand  {\Aut}     {\operatorname{Aut}}

\newcommand  {\can}     {{\text{can}}}

\newcommand  {\Coh}     {\operatorname{Coh}}

\newcommand  {\End}     {\operatorname{End}}

\newcommand  {\Ext}     {\operatorname{Ext}}

\newcommand  {\Fil}     {{\text{\rm {Fil}}}}

\newcommand  {\GL}      {\operatorname{GL}}

\newcommand  {\Hom}     {\operatorname{Hom}}

\newcommand  {\id}      {\operatorname{id}}

\renewcommand  {\ker }  {\operatorname{ker}}

\newcommand  {\kod}     {\operatorname{kod}}

\newcommand  {\invlim}  {\varprojlim}

\newcommand  {\length}  {\operatorname{length}}
\newcommand  {\lra}     {\longrightarrow}

\newcommand  {\Mat}     {\operatorname{Mat}}

\newcommand  {\N}       {\operatorname{N}}
\newcommand  {\naif} 	{\text{{\rm naif}}}

\renewcommand{\O}       {\mathscr{O}}

\newcommand  {\orb}     {\operatorname{orb}}

\newcommand  {\pd}      {\operatorname{pd}}

\newcommand  {\Pic}     {\operatorname{Pic}}

\newcommand  {\pr}      {\operatorname{pr}}
\newcommand  {\Proj}    {\operatorname{Proj}}

\newcommand  {\quadand} {\quad\text{and}\quad}

\newcommand  {\ra}      {\rightarrow}

\newcommand  {\rank}    {\operatorname{rank}}
\newcommand  {\red}     {{\operatorname{red}}}

\newcommand  {\Spec}    {\operatorname{Spec}}

\newcommand  {\Star}    {\operatorname{Star}}

\newcommand  {\Tor}     {\operatorname{Tor}}

\newcommand {\Vect}     {\operatorname{Vec}}

\def\mydate{\number\day\space\ifcase\month \or January\or February\or March\or 
April\or May\or June\or July\or
August\or September\or October\or November\or December\fi \space\number\year}

\DeclareFontFamily{U}{wncy}{}
\DeclareFontShape{U}{wncy}{m}{n}{<->wncyr10}{}
\DeclareSymbolFont{mcy}{U}{wncy}{m}{n}
\DeclareMathSymbol{\Sh}{\mathord}{mcy}{"58}
\begin{document}

\title[Vector bundles on toric 3-folds]
      {Vector bundles on  proper toric 3-folds and certain other   schemes}

\author[Markus Perling]{Markus Perling}
\address{Fakult\"at f\"ur Mathematik, Universit\"at Bielefeld, Postfach 100 131, 33501 Bielefeld, Germany}
%\curraddr{Mathematisches Institut, Heinrich-Heine-Universit\"at, 40204 D\"usseldorf, Germany}
\email{perling@math.uni-bielefeld.de}

\author[Stefan Schr\"oer]{Stefan Schr\"oer}
\address{Mathematisches Institut, Heinrich-Heine-Universit\"at, 40204 D\"usseldorf, Germany}
\email{schroeer@math.uni-duesseldorf.de}


\dedicatory{Revised version, 14 July 2015}

\begin{abstract}
We show that a proper   algebraic $n$-dimensional scheme $Y$
admits nontrivial vector bundles of rank $n$, even if $Y$ is non-projective, provided that there
is a modification containing a projective Cartier divisor 
that intersects the exceptional locus in only finitely many points.
Moreover, there are such vector bundles with arbitrarily large top Chern number.
Applying this to toric varieties, we infer that every proper toric threefold
admits such vector bundles of rank three. Furthermore, we describe a class of higher-dimensional
toric varieties for which the result applies, in terms of convexity properties
around rays.
\end{abstract}

\maketitle
\tableofcontents

%===========================================================
\section*{Introduction}

Let $X$ be an algebraic scheme, that is,
a separated scheme of finite type over a ground field $k$, which is
not necessarily quasiprojective.
A fundamental question is whether or not every coherent sheaf $\shF$ on $X$
is the quotient of some locally free sheaf $\shE$ of finite rank.
If this property holds, one says that $X$ has the \emph{resolution property}.
Totaro \cite{Totaro 2004} gave a characterization of schemes having the resolution property:
They admit some principal $\GL_n$-bundle whose total space is quasiaffine. This should be seen
as a far-reaching generalization of the pointed cone attached to an ample invertible sheaf.

On   schemes $X$ having the resolution property, any coherent sheaf $\shF$ can  be replaced by a complex of locally free
sheaves of finite rank, which has important consequence for K-theory.
If the resolution property is unavailable, one   relies on 
ad hoc approaches, which may become intricate. Here we mention the  definition of Chern classes for coherent sheaves on
arbitrary compact complex manifolds taking values in Deligne cohomology constructed
by Grivaux \cite{Grivaux 2010}.

More generally, the resolution property   makes sense for \emph{algebraic stack}s. 
It is then related to Grothendieck's question on the equality of the Brauer group and the cohomological
Brauer group \cite{Edidin; Hassett; Kresch; Vistoli 1999}. Note that the resolution property does not hold in each and every situation: 
One easily constructs
non-separated schemes without resolution property. A   natural example is the algebraic stack $\shM_0$ of prestable curves
of genus zero, as observed by Kresch \cite{Kresch 2013}.

An even more basic question is whether or not  any proper scheme $X$
admits locally free sheaves $\shE$ of finite rank that are not free, that is $\shE\not\simeq\O_X^{\oplus r}$.
Winkelmann \cite{Winkelmann 1993} showed that this indeed holds for compact complex manifolds.
For proper  schemes, one has  the following   facts:

Any   curve is projective, so there are invertible sheaves $\shL$ with $c_1(\shL)$ arbitrary large.
In contrast, there are normal surfaces $S$ with trivial Picard group, see for example \cite{Schroeer 1999}.
However, any surface admits locally free sheaves $\shE$ of rank $n=2$, in fact
with $c_2(\shE)$ arbitrary large (\cite{Schroeer; Vezzosi 2004}, actually the resolution property holds by results of Gross \cite{Gross 2012}).
Based on these facts, one may arrive at the perhaps over-optimistic conjecture that \emph{any  proper scheme $Y$ should admit 
locally free sheaves $\shE$ of rank $n=\dim(Y)$ with Chern number $c_n(\shE)$ arbitrarily large}.

The main goal of this paper is to provide further bits of evidence for this. Throughout the article, we  assume
that the ground field $k$ is infinite, if not stated otherwise. One of our   results deals with toric varieties of dimension three:

\begin{maintheorem}
Let $Y$ be a proper   toric threefold. Then there are locally free sheaves $\shE$
on $Y$ of rank $n=3$ with arbitrarily large Chern number $c_3(\shE)$.
\end{maintheorem}

In contrast to toric surfaces, smooth proper toric threefolds $Y$ are not necessarily projective.
A characterization of the non-projective ones in terms of triangulations of the 2-sphere was given by Oda 
(\cite{Oda 1978}, Proposition 9.3).
Eikelberg \cite{Eikelberg 1992} gives examples of proper toric threefolds with trivial Picard group, see also
the discussions by Fulton \cite{Fulton 1993}, pp.\ 25--26 and  Ford and Stimets \cite{Ford; Stimets 2002}.

Examples of proper toric threefolds $Y$ whose \emph{toric} vector bundles of rank $\leq 3$ are trivial
were constructed by Payne \cite{Payne 2009}.
In other words, the quotient stack $\shY=[Y/\GG_m^3]$ has
no non-trivial vector bundles of rank $\leq 3$. This result relies on the theory
of branched coverings of cone complexes, together with a computer calculation.
Payne also posed the question whether or not there are nontrivial
vector bundles on $Y$ at all. This question was taken up by Gharib and Karu \cite{Gharib; Karu 2012},
and   our Theorem provides a positive answer to Payne's question.

Note that there has been a strong interest in the $K$-theory of toric varieties in the recent past.
For example,  Anderson and Payne \cite{Anderson; Payne 2013}   showed that for proper toric threefolds
over algebraically closed ground fields, 
the canonical map 
$KH^\circ(X)\ra\text{op}K^\circ(X)$ from the $K$-group of perfect complexes to the
operational $K$-groups of Fulton--MacPherson \cite{Fulton; MacPherson 1981} is surjective.
Gubeladze \cite{Gubeladze 2004} constructed simplicial toric varieties with surprisingly
large $K^\circ(X)$.
We also would like to mention results of Corti\~nas,  Haesemeyer, Walker and Weibel, which express various $K$-groups of toric varieties 
in terms of the cdh-topology (\cite{Cortinas; Haesemeyer; Walker;  Weibel 2009},
\cite{Cortinas; Haesemeyer; Walker;  Weibel 2014}). 

Our theorem above is actually a simple consequence of the following   more general statement, 
which is the main result of this paper:

\begin{maintheorem}
Let $Y$ be a proper scheme.
Suppose there is a proper birational morphism $X\ra Y$ and a
Cartier divisor $D\subset X$ that intersects the exceptional locus  
in a finite set, and that the proper scheme $D$ is projective. 
Then there are locally free sheaves $\shE$ of rank $n=\dim(Y)$ on $Y$ with Chern number $c_n(\shE)$ arbitrarily large.
\end{maintheorem}

Indeed, this   is a generalization to higher dimensions of a result  of the second author and  Vezzosi \cite{Schroeer; Vezzosi 2004}
on proper surfaces, where the assumptions are vacuous. It would be interesting to find examples
in dimension $n\geq 3$ where all vector bundles of rank $\leq n-1$ are trivial.

This result also has   applications to toric varieties in arbitrary dimension $n\geq 3$:
Indeed, we give characterizations for proper toric $n$-folds $Y$ so that there is a \emph{toric} modification
$f:X\ra Y$ and a \emph{toric} divisor $D\subset X$ satisfying the assumptions of our main result,
in terms of convexity properties around the ray $\rho$ corresponding to the Weil divisor $f(D)\subset Y$
in the fan $\Delta$ that describes the toric variety $Y=Y_\Delta$. Roughly speaking, any cone $\sigma\in \Star(\rho)$
has to be a \emph{pyramidal extension} of the cone $\sigma'$ generated by the other rays $\rho'\neq\rho$ in $\sigma$.
The notion of pyramidal extensions is closely related to the so-called \emph{beneath-and-beyond method}
of convex geometry, and leads to a condition on the ray $\rho\in\Delta$ which we choose to  call \emph{in Egyptian position}.
In dimension $n\leq 3$, any ray is in Egyptian position, but the condition becomes nontrivial in higher dimensions.

\medskip
The paper is organized as follows:
We start in Section \ref{Infinitesimal Neighborhoods} by showing that under certain circumstances,
locally free sheaves   are determined on infinitesimal neighborhoods of exceptional sets.
This is used in Section \ref{Equivalence} to deduce an equivalence of categories
between locally free sheaves on $Y$ and certain proper $Y$-schemes $X$.
Our main theorem appears in Section \ref{Elementary Transformations}, 
in which we construct
infinitely many locally free sheaves $\shE_t$ on certain proper schemes.
To see that these sheaves have unbounded top Chern number, we  
investigate in Section \ref{Chern Classes} Chern classes for   coherent sheaves admitting short global resolutions,
without the usual assumption that any coherent sheaf is the quotient 
of a locally free sheaf.
In Section \ref{Toric Varieties}, we apply our result to toric varieties, and in particular
to toric threefolds. Section \ref{Examples trivial} contains concrete examples
of toric varieties with trivial Picard group for which our results apply.
The final Section \ref{Divisors on threefolds} contains a sufficient condition for  certain proper threefolds
to  contain projective divisors.

\begin{acknowledgement}
The authors wish to thank the referee for a thorough and careful report, which helped 
to  remove some mistakes, to clarify certain arguments, and to improve the overall exposition.
\end{acknowledgement}

%===========================================================
\section{Vector bundles on infinitesimal neighborhoods}
\mylabel{Infinitesimal Neighborhoods}

In this section, we study the behavior of locally free sheaves near certain closed fibers.
The set-up is as follows:
Suppose $R$ is a local noetherian ring,   denote by $\idealm=\idealm_R$ its maximal ideal, 
and let $y\in \Spec(R)$ be the corresponding closed point.
Let $f:X\ra\Spec(R)$ be a proper morphism, and $X_y=f^{-1}(y)$ be the closed fiber.
For each coherent sheaf $\shF$ on $X$, we regard the cohomology groups $H^p(X,\shF)$
as $R$-modules, which are finitely generated, though not of finite length in general.
In what follows, we shall consider certain \emph{infinitesimal neighborhoods} of the reduced closed fiber,
that is, closed subschemes $E\subset X$ having the same 
topological space as $f^{-1}(y)\subset X$.

\begin{theorem}
\mylabel{isomorphic  restriction}
Let $\shE$ be a locally free sheaf of finite rank on $X$.
Suppose that the local noetherian ring $R$ is complete, and that 
the $R$-modules $H^1(X,\shE)$ and $H^2(X,\shE)$ have finite length. Then there exists
an infinitesimal neighborhood $ E $ of the reduced closed fiber 
with the following property: For any   locally free sheaf  $\shF$ on $X$
with $\shF|E\simeq \shE|E$, we already have $\shF\simeq \shE$.
\end{theorem}

\proof
Set $\shI=\idealm\O_X$. Let $E_k\subset X$
be the closed subscheme corresponding to $\shI^{k+1}\subset\O_X$, which are infinitesimal neighborhoods of the
closed fiber $X_y=E_0$. We have inclusions of subschemes $E_0\subset E_1\subset \ldots$,
and the subsheaf $\shI^{k+1}/\shI^{k+2}\subset\O_X/\shI^{k+2}$ is the ideal sheaf for  
$E_{k}\subset E_{k+1}$.
Clearly, $\shI$ annihilates the coherent sheaf $\shI^{k+1}/\shI^{k+2}$, hence
the   schematic support $A_{k+1}\subset X$ of the latter is contained in $E_0$.

Suppose for the moment that we already know that  there is an integer $m\geq 0$ such that the groups
$H^1(X,\underline{\End}(\shE_{A_{k+1}})\otimes_{\O_{A_{k+1}}} \shI^{k+1}/\shI^{k+2}) $, which coincide 
with 
\begin{equation}
\label{first cohomology}
H^1(X,\underline{\End}(\shE)\otimes\shI^{k+1}/\shI^{k+2}) = H^1(X,\shI^{k+1}\underline{\End}(\shE)/\shI^{k+2}\underline{\End}(\shE)), 
\end{equation}
vanish for all $k\geq m$. We now check that $E=E_m$ has  the desired property:
Let   $\shF$ be  a locally free sheaf   on $X$ with  $\shF|E_m\simeq \shE|E_m$.
In light of Corollary \ref{isomorphic increase}, which for the sake of the exposition is deferred to the end of this section, it follows
 by induction  that $\shF|E_k\simeq\shE|E_k$ for all $k\geq m$.

In turn, the isomorphism classes of $\shE$ and $\shF$ have the same image under the canonical map 
$$
H^1(X,\GL_r(\O_X))\lra \invlim_k H^1(E_k,\GL_r(\O_{E_k})).
$$
Let $\foX$ be the formal completion of $X$ along the closed fiber. 
As explained in \cite{Artin 1969}, proof for Theorem 3.5, the canonical map
$$
H^1(\foX,\GL_r(\O_\foX))\lra \invlim_k H^1(E_k,\GL_r(\O_{E_k})).
$$
is bijective. In other words, the sheaves $\shE$ and $\shF$ are formally isomorphic.
Since the local noetherian ring $R$ is complete, we may apply the Existence Theorem
(\cite{EGA IIIa}, Theorem 5.1.4) and conclude that $\shE$ and $\shF$ are isomorphic.

It remains to verify that  the groups (\ref{first cohomology}) indeed vanish for all $k$ sufficiently large.
This is a special case of the next assertion.
\qed

\begin{proposition}
\mylabel{cohomology vanishes}
Let $\shF$ be a coherent sheaf on our scheme $X$ and $p\geq 1 $ an integer 
such that the $R$-modules $H^p(X,\shF)$ and $H^{p+1}(X,\shF)$
have finite length.
Then there is an integer $m\geq 0$ so that $H^p(X,\idealm^k\shF/\idealm^{k+1}\shF)=0$ for all $k\geq m$.
\end{proposition}

\proof
Consider the Rees ring
$S=\bigoplus _{k\geq 0} \maxid^k$ corresponding to the $\maxid$-adic filtration on $R$,
which has invariant subring $S_0=R$ and irrelevant ideal $S_+=\bigoplus_{i\geq 1}\maxid^k$.
The graded $S$-module 
$$
\bigoplus_{k\geq 0} H^p(X,\maxid^k\shF)
$$
is finitely generated, according to the Generalized Finiteness Theorem 
(\cite{EGA IIIa}, Corollary 3.3.2).
In particular, there is an integer $n\geq 1$ such  that
$H^p(X,\maxid^{n+i}\shF)= \maxid^i H^p(X,\maxid^n\shF)$ for all $i\geq 0$.
The short exact sequence
$$
0\lra\maxid^n\shF\lra\shF\lra\shF/\maxid^n\shF\lra 0
$$
yields an exact sequence
$$
H^{p-1}(X,\shF/\maxid^n\shF)\lra H^p(X,\maxid^n\shF)\lra H^p(X,\shF)
$$
of finitely generated $R$-modules. The term on the right has finite length
by assumption, and the term on the left is annihilated by $\maxid^n\subset R$, thus
has finite length as well. So the term in the middle has finite length.
In particular, $H^p(X,\maxid^n\shF)$ is   annihilated by
$\maxid^d$ for some integer $d\geq 0$. Consequently,
$H^p(X,\maxid^{n+i}\shF)=0$ for all $i\geq d$.

The same arguments apply in degree $p+1$ instead of $p$, and we thus have shown
that there is an integer $m\geq 0$ such that
 $H^p(X,\maxid^k\shF)$ and $H^{p+1}(X,\maxid^k\shF)$ vanish for all $k\geq m$.
The short exact sequence 
$$
0\lra\maxid^{k+1}\shF\lra\maxid^k\shF\lra\maxid^k\shF/\maxid^{k+1}\shF\lra 0
$$
of sheaves yields an exact sequence 
$$
H^p(X,\maxid^k\shF)\ra H^p(X,\maxid^k\shF/\maxid^{k+1}\shF)\ra H^{p+1}(X,\maxid^{k+1}\shF) 
$$
of $R$-modules, and it follows  that  $H^p(X,\maxid^k\shF/\maxid^{k+1}\shF)$ vanishes for $k\geq m$.
\qed

\begin{remark}
Theorem \ref{isomorphic  restriction} remains  true if $R$ is the henselization of a ring $A$ with respect to some prime ideal,
provided that $A$ is finitely generated over some field or some excellent Dedekind ring. 
Indeed, by  \cite{Artin 1969}, Theorem 3.5 the restriction map 
$$
H^1(X,\GL_r(\O_X))\ra H^1(\foX,\GL_r(\O_\foX))
$$ 
is injective, which relies on Artin's Approximation Theorem \cite{Artin 1969}, Theorem 1.12. In light of Popescu's generalization 
\cite{Popescu 1985}, Theorem 1.3 (see also the discussions of Swan \cite{Swan 1998} and  Conrad and de Jong \cite{Conrad; de Jong 2002}),
it remains valid under the assumption that  $R$ is any henselian excellent local ring.
\end{remark}

\begin{remark}
\mylabel{no torsion}
We may assume that the structure sheaf $\O_E$ of the infinitesimal neighborhood $f^{-1}(y)_\red\subset E$
contains no nonzero local sections whose support is finite. Indeed, if $\shJ\subset\O_E$
is the ideal of such local sections, then we have $H^1(E,\underline{\End}(\shE)\otimes\shJ)=0$, and Corollary \ref{isomorphic increase} below
tells us that a locally free sheaf $\shF$ that with 
$\shF|{E'}\simeq\shE|{E'}$ already has $\shF|E\simeq\shE|E$.
\end{remark}

\begin{remark}
The proof for Theorem \ref{isomorphic  restriction} reveals that one  may choose the infinitesimal neighborhood as
$E=X_y$ provided
that the groups (\ref{first cohomology}) vanish for all $k\geq 0$. However, it appears difficult
to give a natural interpretation for this condition if the ideal sheaf $\shI=\maxid\O_X$
is not invertible.
\end{remark}

\medskip
In the proof for Theorem \ref{isomorphic   restriction}, we have used Corollary \ref{isomorphic increase} below,
and we now gather the necessary facts from non-abelian cohomology.
Let $X$ be a scheme,  $\shI\subset \O_X$ be a quasicoherent ideal sheaf with $\shI^2=0$,
and   $X'\subset X$ be the corresponding closed subscheme.
Let $\shE$ be a locally free sheaf  of finite rank on $X$, and $\shE'=\shE_{X'}=\shE\otimes_{\O_X}\O_{X'}=\shE/\shI\shE$
its restriction to $X'$. Each homomorphism $f:\shE\ra\shE$ necessarily has $f(\shI\shE)\subset\shI\shE$,
therefore induces a map $f':\shE'\ra\shE'$. We thus obtain a homomorphism
of group-valued sheaves
$$
\underline{\Aut}(\shE)\lra\underline{\Aut}(\shE'),\quad f\longmapsto f'.
$$
Furthermore, each homomorphism $h:\shE\ra\shI\shE$  yields a homomorphism
$$
\shE\lra\shE,\quad s\longmapsto s+h(s).
$$
Using $\shI^2=0$, we see that $s\mapsto s-h(s)$ is an inverse, and  that the resulting mapping
$\underline{\Hom}(\shE,\shI\shE)\ra\underline{\Aut}(\shE)$
is a homomorphism  of group-valued sheaves. We thus obtain a sequence
\begin{equation}
\label{automorphism sheaves}
0\lra \underline{\Hom}(\shE,\shI\shE)\lra\underline{\Aut}(\shE) \lra \underline{\Aut}(\shE')\lra 1
\end{equation}
 of group-valued sheaves. Note that the term on the left is commutative  and written additively,
whereas the other terms  are in general non-commutative,  and written multiplicatively.

\begin{lemma}
\mylabel{exact}
The sequence (\ref{automorphism sheaves}) is exact.
\end{lemma}

\proof
This is a local problem, so it suffices to treat the case that $\shE=\O_X^{\oplus r}$ is free.
It  follows immediately from the definition of the maps that the sequence is a complex.
Now let $x\in X$ be a point, $x\in U\subset X$ an open neighborhood, and $A'\in \GL_r(\Gamma(X,\O_{X'}))$.
Shrinking $U$, we may lift the entries of the invertible matrix $A'$ and obtain
a matrix $A\in\Mat_r(\Gamma(X,\O_X))$. Then $\det(A)$ is a unit, because it is a unit modulo
the nilpotent ideal $\Gamma(X,\shI)$. Hence the complex is exact at the term on the right.
An element $A\in\Mat_r(\Gamma(X,\O_X))$ mapping to the identity matrix in $\GL_r(\Gamma(X,\O_{X'}))$
differs from the identity matrix by  some $h\in\Mat_r(\Gamma(X,\shI))$, so the complex
is exact in the middle.
Since hom functors are left exact in the second variable,
the induced map $\underline{\Hom}(\shE,\shI\shE)\ra \underline{\Hom}(\shE,\shE)$ is injective,
therefore the corresponding map to $\underline{\Aut}(\shE)$ is injective as well.
\qed

\medskip
One may simplify the term on the left in the short exact sequence (\ref{automorphism sheaves})
in rather general circumstances.
In what follows, we tacitly suppose that
 that the kernel $\shN$ of the canonical homomorphism $\O_X\ra\underline{\End}(\shI)$ is quasicoherent,
which holds, for example, if $\shI$ is of finite presentation (\cite{EGA I}, Corollary 2.2.2), 
and in particular if $X$ is locally noetherian.
Let $A\subset X$ be the corresponding closed subscheme, which is called the
\emph{schematic support} for the sheaf $\shI$. Note that the $\O_X$-module $\shI$ is actually an $\O_A$-module.
Since $\shI^2=0$, we have $\shN\supset\shI$, hence $A\subset X'$.

\begin{lemma}
\mylabel{identification}
There is a canonical identification 
$$
\underline{\End}_{\O_A}(\shE_A)\otimes_{\O_A}\shI=\underline{\Hom}_{\O_X}(\shE,\shI\shE)
$$
of quasicoherent $\O_X$-modules
\end{lemma}

\proof
The ideal sheaf
$\shN\subset\O_X$ annihilates $\shI\shE$, hence  the canonical injection 
$$
\underline{\Hom}_{\O_A}(\shE/\shN\shE,\shI\shE) \lra\underline{\Hom}_{\O_X}(\shE,\shI\shE) 
$$
is bijective. 
Using the   identifications $\shE/\shN\shE=\shE_A$ and $\shI\shE=\shI\otimes_{\O_A}\shE_A$
and the fact that $\shE_A$ is locally free on $A$, we obtain an identification
$$
\underline{\Hom} (\shE/\shN\shE,\shI\shE)=
\shE_A^\vee\otimes \shE_A\otimes\underline{\Hom} (\O_A,\shI)=
\underline{\End} (\shE_A)\otimes\shI,
$$
where all tensor products and hom sheaves are over $\O_A$.  
\qed

\medskip
Now recall that $X'\subset X$ is a closed subscheme whose ideal $\shI$ has square zero,
and $A\subset X'$ is the closed subscheme whose ideal is the
kernel of $\O_X\ra\underline{\End}(\shI)$. Under these assumptions, we get the following result by
combining the preceding lemmas:

\begin{proposition}
\mylabel{matrix exact}
There is a short exact sequence
$$
0\lra\underline{\End}_{\O_A}(\shE_A)\otimes_{\O_A}\shI\lra\underline{\Aut}(\shE) \lra \underline{\Aut}(\shE_{X'})\lra 1
$$
of group-valued sheaves.
\end{proposition}

This has the following consequence, which was used in a crucial way for the proof for Theorem \ref{isomorphic  restriction}:

\begin{corollary}
\mylabel{isomorphic increase}
Assumptions as above. If  $H^1(X,\underline{\End}_{\O_A}(\shE_A)\otimes_{\O_A}\shI)=0$,
then a locally free sheaf  $\shF$
on $X$ is isomorphic to $\shE$ if and only if   $\shF_{X'}\simeq\shE_{X'}$.
\end{corollary}

\proof
The short exact sequence of group valued sheaves in the Proposition yields an exact sequence
$$
H^1(X,\underline{\End}_{\O_A}(\shE_A)\otimes_{\O_A}\shI)\lra H^1(X,\underline{\Aut}(\shE))\lra H^1(X,\underline{\Aut}(\shE_{X'}))
$$
of pointed sets, by the machinery of non-abelian cohomology exposed  in \cite{Giraud 1971}, 
Chapter III, \S3.
The term in the middle is the set of isomorphism classes of $\O_X$-modules that
are locally isomorphic to $\shE$, which coincides with the set of isomorphism classes of locally free
sheaves of rank $r=\rank(\shE)$.
Exactness means that the image of the map on the left is the set of isomorphism classes
whose restrictions to $X'$ become isomorphic to $\shE_{X'}$.
By assumption,  the term on the left  
consists of a single element.
\qed

%===========================================================
\section{An equivalence of categories}
\mylabel{Equivalence}

Let $X$ be a scheme. We denote
by $\Vect(X)$ the exact category of locally free sheaves of finite rank on $X$.
Given a closed subscheme $E\subset X$, we write 
$\Vect_E(X)\subset\Vect(X)$
for the full subcategory of all locally free sheaves $\shE$ on $X$
whose restriction to $E$ is free, 
that is, $\shE|E \simeq\O_E^{\oplus r}$, with $r=\rank(\shE)$.
Note that if $X$ is not connected, one has to regard the rank as a locally constant
function $x\mapsto\rank_x(\shE)$.
More generally, if $\Phi=\left\{E_\alpha\right\}_{\alpha\in I}$ is a collection of closed subschemes
(``family of supports''),
we denote by 
$$
\Vect_\Phi(X)\subset \Vect(X)
$$ 
the full subcategory of the $\shE$ that become free on each $E_\alpha\in\Phi$.
                      
Now let $Y$ be a noetherian scheme, and $f:X\ra Y$ be a proper morphism
with $\O_Y=f_*(\O_X)$. Suppose that
the coherent sheaves $R^1f_*(\O_X)$ and $R^2f_*(\O_X)$ have finite supports.
Applying Theorem \ref{isomorphic restriction}, we choose for each closed point 
$y\in Y$ with $\dim f^{-1}(y)\geq 1$ an infinitesimal neighborhood $f^{-1}(y)_\red\subset E_y$
so that  locally free sheaves of finite rank on $X\otimes_{\O_{Y}}\O_{Y,y}^\wedge$
that become  free on $E_y$ are already free. Let $\Phi$ be the collection of these $E_y$.
We thus obtain a functor $f^*:\Vect(Y)\ra\Vect_\Phi(X)$.

\begin{theorem}
\mylabel{equivalence}
Let $f:X\ra Y$ be a proper morphism of noetherian schemes with $\O_Y=f_*(\O_X)$  such that $R^1f_*(\O_X)$ and $R^2f_*(\O_X)$ have finite supports,
and let $\Phi$ be the collection of closed subschemes defined above.
Then for every $\shE\in\Vect_\Phi(X)$, the coherent $\O_Y$-module $f_*(\shE)$ is locally free, 
and the functors 
$$
f^*:\Vect(Y)\lra\Vect_\Phi(X)\quadand f_*:\Vect_\Phi(X)\lra\Vect(Y)
$$ 
are quasi-inverse equivalences of categories.
\end{theorem}

\proof
Suppose first that $Y=\Spec(R)$ is the spectrum of a complete local noetherian ring,
with  closed point $y\in Y$. The assertion is trivial if the closed fiber is zero-dimensional,
because then $f:X\ra Y$ is an isomorphism.  Suppose now that $\dim f^{-1}(y)\geq 1$.
Let $\shE$ be a locally free sheaf of finite rank on
$X$ whose restriction to $E_y\subset X$ is free. This implies, by the choice of $E_y$,
that $\shE$ is free. Using the assumption $\O_Y= f_*(\O_X)$, we infer that
$f_*(\shE)$ is free. To see that $f^*f_*$ and $f_*f^*$ are isomorphic to the respective identity functors,
it thus suffices to verify this for the structure sheaves $\O_X$ and $\O_Y$, which
again follows from $\O_Y= f_*(\O_X)$. Thus the assertion holds in this special case.

We now come to the general case. Let $\shE\in\Vect_\Phi(X)$.
To verify that the coherent sheaf $f_*(\shE)$ is locally free, it suffices to check that its stalks
at closed points are free. Fix a closed point $y\in Y$. In order to check that $f_*(\shE)_y$
is free, we may assume that $Y=\Spec(R)$ is the spectrum of a local ring.
By faithfully flat descent (see \cite{SGA 1}, Expose VIII, Corollary 1.11),
it suffices to treat the case that $R$ is complete, which indeed
holds by the preceding paragraph.
Summing up, for each $\shE\in\Vect_\Phi(X)$, the coherent sheaf $f_*(\shE)$ is locally free. 

To see that
the natural adjunction map
$f^*(f_*(\shE))\ra \shE$
is bijective for each $\shE\in\Vect_\Phi(X)$, it again suffices to treat the case that $Y$ is the spectrum of a complete
local noetherian ring. It then follows that $\shE$ is free,
and bijectivity follows from $\O_Y= f_*(\O_X)$. 
Finally, checking the bijectivity of the  natural adjunction map
$\shF\ra f_*(f^*(\shF))$ with $\shF\in\Vect(Y)$   is a local problem, so we may 
assume that $\shF$ is free, and then  conclude with $\O_Y= f_*(\O_X)$.
\qed

\begin{remark}
Suppose that $Y$ is normal and admits a resolution of singularities,
and that $f:X\ra Y$ is proper and birational. The assumption that $R^1f_*(\O_X)$
and $R^2f_*(\O_X)$ have finite support holds in particular if the local schemes
$\Spec(\O_{Y,y})$ have only rational singularities, for all nonclosed
points $y\in Y$.
\end{remark}

\medskip
We are mainly interested in the following situation: Suppose
that $Y$ is a noetherian scheme  and    $f:X\ra Y$ be a proper morphism 
with $\O_Y=f_*(\O_X)$.
Let $E\subset X$ be the \emph{exceptional locus}, that is, the set of points $x\in X$ where $\dim_xf^{-1}(f(x))>0$,
which is closed by Chevalley's Semicontinuity Theorem (\cite{EGA IVc}, Theorem 13.1.3). 
This $E$ can be viewed as  the union of all irreducible curves
mapping to points. Its image $f(E)\subset Y$, which is a closed set, is called the \emph{critical locus}.

The exceptional locus $E\subset X$ can also be regarded as the set of points where
the morphism $f:X\ra Y$ is ramified in the sense of \cite{EGA IVd}, Definition 17.3.1.
Thus there is a canonical scheme structure on $E$,   being the  support of the coherent sheaf $\Omega^1_{X/Y}$.
In the following, however, we shall regard the exceptional locus either as a closed subset,
or  choose an infinitesimal neighborhood that makes the exceptional locus large enough in the following sense:

\begin{corollary}
\mylabel{critical finite}
Let $Y$ be a noetherian scheme, $f:X\ra Y$ a proper morphism with $\O_Y=f_*(\O_X)$,
whose critical locus  is finite. Then there is an infinitesimal
neighborhood $E$ of the exceptional locus with the following property:
For every $\shE\in\Vect_E(X)$, the coherent sheaf $f_*(\shE)$ is locally free, 
and the functors 
$$
f^*:\Vect(Y)\lra\Vect_E(X)\quadand f_*:\Vect_E(X)\lra\Vect(Y)
$$ 
are quasi-inverse equivalences of categories.
\end{corollary}

\proof
Let $y_1,\ldots,y_r\in Y$ be the points comprising the critical locus $C\subset Y$.
Clearly, the coherent sheaves $R^pf_*(\O_X)$, $p\geq 1$ are supported by the finite set $C$.
Moreover, the $f^{-1}(y_i)$, $1\leq i\leq r$ are precisely the fibers that are not zero-dimensional,
and their union is the the exceptional locus for $f:X\ra Y$. We then choose infinitesimal neighborhoods
$f^{-1}(y_i)\subset E_{y_i}$ as for the Theorem, and for the union $E=E_{y_1}\cup\ldots\cup E_{y_r}$
the assertion follows.
\qed

%===========================================================
\section{Elementary transformations}
\mylabel{Elementary Transformations}

Fix an infinite  ground field $k$,  let $Y$ be a proper scheme, and $f:X\ra Y$ be a proper birational morphism with $\O_Y=f_*(\O_X)$.
Let $E\subset X$ be the exceptional locus.  

\begin{theorem}
\mylabel{infinitely classes}
Assumptions as above. Suppose the following three conditions:
\begin{enumerate}
\item
The critical locus $f(E)\subset Y$ is finite.
\item
There is an effective Cartier divisor $D\subset X$ such that $D\cap E$ is finite.
\item
The proper scheme $D$ is projective.
\end{enumerate}
Then there are infinitely many isomorphism classes of locally free sheaves   on $Y$ of rank $n=\dim(Y)$.
\end{theorem}

\proof
Let us first discuss the case $\dim(D)=0$.  Then   each irreducible component $X'\subset X$
intersecting $D$ is 1-dimensional. Since $X\ra Y$ is birational, $X'\cap E$ is finite. It follows that
the irreducible component $C=f(X')$ of $Y$ is a curve. Since proper curves are projective,
and every invertible sheaf on $C$ can be represented by
a Cartier divisor whose support is disjoint from the closure of $Y\smallsetminus C\subset Y$, one easily
sees that the restriction map $\Pic(Y)\ra\Pic(C)$ is surjective. Hence
the locally free sheaves on $Y$ of the form $\shE=\shL^{\oplus n}$, with $\shL$ invertible and $(\shL\cdot C)=\deg(\shL_C)>0$
yield the assertion.

Suppose now that $\dim(D)\geq 1$.
According to Corollary \ref{critical finite}, we may choose a suitable infinitesimal neighborhood  $E$ of
the exceptional set so
that
the pullback functor $f^*$ induces an equivalence  between the category $\Vect(Y)$ of locally free sheaves
of finite rank on $Y$ and the category $\Vect_E(X)$ of locally free sheaves on $X$ whose
restriction to $E$ are free. This allows us to work entirely on $X$ rather than $Y$.
In light of Remark \ref{no torsion},
we additionally may assume that the structure sheaf $\O_E$ contains no nontrivial local   sections  
supported on the finite set $D\cap E$.

Let  $\shE$ be some locally free sheaf of rank $n=\dim(Y)=\dim(X)$ on $X$ that   becomes free on $E$.
For example, we could take the free sheaf $\shE=\O_X^{\oplus n}$.
The following construction   yields other locally free sheaves $\shE'$ on $X$ that become free on $E$.

To start with, fix an ample invertible sheaf $\O_D(1)$.
For simplicity, we write $\shE_D=\shE|D$ for the induced
locally free sheaf on the effective Cartier divisor $D\subset X$. 
By Proposition \ref{schematic support} below applied to the projective scheme $D$, which has dimension $\leq n-1$,
and the empty closed subscheme $A=\emptyset$, 
there is some integer $s\geq 1$ so that there exists a surjection $\shE_D\ra\O_D(s)$. Note that
here our assumption that the ground field $k$ is infinite enters. Composing with the
canonical projection, we get a  surjection $\shE\ra\O_D(s)$.
The short exact sequence
$$
0\lra\shF\lra\shE\lra\O_D(s)\lra 0
$$
defines a coherent sheaf $\shF$ on $X$, which is locally free because the stalks of $\O_D(s)$ have 
projective dimension $\leq 1$ as modules over the stalks of $\O_X$. Such    $\shF$ are called
\emph{elementary transformations} of $\shE$. One may recover the latter from the former: Dualizing the preceding
short exact sequence yields
$$
0\lra\shE^\vee\lra\shF^\vee\lra\underline{\Ext}^1(\O_D(s),\O_X)\lra 0.
$$
This is exact, because the sheaves $\underline{\Hom}(\O_D(s),\O_X)$ and  $\underline{\Ext}^1(\shE,\O_X)$ vanish.

Now we view
$\underline{\Ext}^p(\O_D(s),\shM)$ and $\underline{\Ext}^p(\O_D,\shM)\otimes\O_D(-s)$
as $\delta$-functors in $\shM$. Obviously, both are exact and  vanish on injective $\O_X$-modules $\shM$,
hence are universal. Moreover, 
we have a canonical bijection $\underline{\Hom}(\O_D,\shM)\otimes\O_D(-s)\ra \underline{\Hom}(\O_D(s),\shM)$
given by $h_1\otimes h_2\mapsto h_1\circ h_2$, where we regard the local section $h_2$ of $\O_D(-s)$ as a homomorphism $\O_D(s)\ra \O_D$. 
In turn, our two universal $\delta$-functors are isomorphic (\cite{Grothendieck 1957}, Section 2.1).

Using the resulting identification
 $\underline{\Ext}^1(\O_D(s),\O_X)=\shN_D(-s)$, where  $\shN_D=\O_D(D)$ be the \emph{normal sheaf} of the effective Cartier divisor,
which is an invertible sheaf on $D$,
we rewrite the preceding short exact sequence as
\begin{equation}
\label{first sequence}
0\lra\shE^\vee\lra\shF^\vee\lra\shN_D(-s)\lra 0,
\end{equation}
and denote the surjective map on the right by  $\phi:\shF^\vee\ra\shN_D(-s)$. Now suppose that
$\psi:\shF^\vee\ra\shN_D(t-s)$ is another surjection for some integer $t\geq 0$. 
Then the short exact sequence
\begin{equation}
\label{second sequence}
0\lra \shE'^\vee\lra\shF^\vee\stackrel{\psi}{\lra}\shN_D(t-s)\lra 0
\end{equation}
defines a new locally free sheaf $\shE'=\shE'_{t,\psi}$ of rank $n$, whose dual is isomorphic to $\ker(\psi)$.
 In the special case $t=0$ and $\psi=\phi$,
we obviously have $\shE'=\shE$. In general, however, the   exact sequences 
(\ref{first sequence}) and (\ref{second sequence}) yield
$$
\chi(\shE'^\vee)- \chi(\shE^\vee)= P(t-s) - P(-s),
$$
where $P(t)=\chi(\shN_D(t))=\sum_i(-1)^ih^i(\shN_D(t))$ is the Hilbert polynomial of the invertible sheaf $\shN_D$ on $D$
with respect to the ample   sheaf $\O_D(1)$. This Hilbert polynomial has degree $\dim(D)\geq 1$,
hence  $P(t-s)-P(-s)$ is a nonzero polynomial in $t$, of the same degree. It follows that the locally free sheaves
$\shE'=\shE'_{t,\psi}$ are pairwise non-isomorphic for $t$ sufficiently large.

It remains to choose $t\geq 0$ and $\psi$ so that the locally free sheaf $\shE'$ becomes free on the closed subscheme $E\subset X$.
Restricting the short exact sequence of sheaves (\ref{first sequence}) to the subscheme $E$ one obtains a short exact sequence
$$
\underline{\Tor}_1^{\O_X}(\O_E,\shN_D(-s))\lra\shE^\vee_E\lra\shF^\vee_E\lra\shN_D(-s)| E\lra 0.
$$
The map  on the left vanishes, because the tor sheaf is supported by the finite set $D\cap E$,
and   $\shE^\vee_E$, regarded as a locally free sheaf on $E$, has no local sections supported
by $D\cap E$. The latter holds, because we have arranged things so that $\O_E$ contains no such local sections.
The upshot is that we get a short exact sequence
$$
0\lra\shE^\vee_E\lra\shF^\vee_E\stackrel{\phi_E}{\lra}\shN_D(-s)| E\lra 0.
$$
Note that the term on the right is supported by 
$D\cap E$, which is finite.  From (\ref{second sequence}) we   likewise get a short exact sequence
$$
0\lra\shE'^\vee_E\lra\shF^\vee_E\stackrel{\psi_E}{\lra}\shN_D(t-s)| E\lra 0.
$$
Now suppose that $t\geq 0$ and $\psi$ are chosen so that there is an isomorphism of skyscraper sheaves 
$h:\shN_D(-s)| E\ra \shN_D(t-s)| E$
with $\psi_E=h\circ\phi_E$. It then follows that $\shE'^\vee_E=\ker(\psi_E)$
is isomorphic to $\shE^\vee_E=\ker(\phi_E)$, hence the restriction $\shE'_E\simeq\shE_E\simeq\O_E^{\oplus n}$ is free.

This can be achieved as follows:
Choose an integer $t_0\geq 0$ so that $\O_D(t)$ is globally generated  for all $t\geq  t_0$.
According to Proposition \ref{schematic support} below, there is an integer $t_1\geq 0$ so that
for all $t\geq t_1$, there is a homomorphism $\shF^\vee_D\ra\shN_D(t-s)$ whose cokernel
has  as schematic support an infinitesimal neighborhood of $D\cap E\subset D$.
Note that here again the assumption that the ground field $k$ is infinite enters.

Now suppose that we have an integer $t$ satisfying $t\geq\max(t_0,t_1)$.
First, choose a homomorphism $\shF^\vee_D\ra\shN_D(t-s)$ as above, and regard it
as a  homomorphism $h:\shF^\vee\ra\shN_D(t-s)$ whose cokernel has an infinitesimal neighborhood of 
$D\cap E\subset D$ as schematic support.
Thus the base-change of $h$ to $ E$ yields an exact sequence
$$
\shF^\vee_{E}\stackrel{h_{E}}{\lra}\shN_D(t-s)|E\lra\O_{D\cap E}\lra 0.
$$
The two terms on the right are invertible $\O_{D\cap E}$-modules,
therefore the surjection on the right is bijective, such that $h_E=0$.
Second, choose a  global section $g\in H^0(D,\O_D(t))$ that does not vanish
at the finite subset $D\cap E$, and regard it as a map $g:\O_X\ra\O_D(t)$,
which is surjective at $D\cap E$. Now consider the homomorphism
$$
\psi=\phi\otimes g + h:\shF^\vee\lra\shN_D(t-s).
$$
On the subsheaf $\shE^\vee\subset\shF^\vee$, the map $\phi$ obviously vanishes and  the 
map $\psi$ coincides with $h$. The latter is surjective outside the
subscheme $D\cap E\subset X$. Base-changed to the subscheme $E\subset X$, the map $h$ vanishes and the
map $\psi$ coincides with $\phi\otimes g$, which is surjective at all points of $D\cap E$.
The upshot is that $\psi$ is surjective and thus qualifies for our construction:
its kernel is locally free of rank $n$, hence of the form  $\shE'^\vee$ for some locally free sheaf $\shE'$.

We  just saw   that the base-change $\psi_E:\shF^\vee_E\lra\shN_D(t-s)|E$ coincides with $(\phi\otimes g)_E$.
Thus there is a commutative diagram
$$
\begin{CD}
\shF^\vee_E @>\phi_E>> \shN_D(-s)|E  \\
@V\id VV @VV\id\otimes gV\\
\shF^\vee_E @>>\psi_E> \shN_D(t-s)|E.
\end{CD}
$$
As discussed above, this implies that $\shE'_E\simeq\shE_E\simeq\O_{E}^{\oplus n}$.
The upshot is that for each $t\geq\max(t_0,t_1)$, we indeed found some $\psi$ giving a  locally free sheaf
$\shE'=\shE'_{t,\psi}$ on $X$ whose restriction to $E$ is free. This yields infinitely many
isomorphism classes of locally free sheaves $\shE'$ on $X$ of rank $n$ that are free on $E\subset X$.
\qed

\medskip
In the preceding proof, we have used the following fact:

\begin{proposition}
\mylabel{schematic support}
Let $X$ be a quasiprojective scheme, $\O_X(1)$ an ample invertible sheaf, $\shE$ a locally free sheaf of finite rank $r>\dim(X)$,
and $A\subset X$ a finite closed subscheme. Then there is an integer $t_0$
so that for all $t\geq t_0$, there is a   homomorphism $\shE\ra\O_X(t)$ such that the
schematic support of the cokernel is an infinitesimal neighborhood of $A$.
\end{proposition}

\proof
We first reduce to the case that $X$ is projective:
Choose an embedding $X\subset\PP^n$,   consider the schematic closure $X'=\bar{X}$, and extend $\shE$
to a coherent sheaf $\shE'$ on $X'$. According to a  result of Moishezon (\cite{Moishezon 1969}, Lemma 3.5, see also
\cite{Bogomolov; Landia 1990}), there is a blowing-up
$X''\ra X'$  so that the strict transform $\shE''$ of $\shE'$ becomes
locally free. Moreover, we may assume that the center of the blowing-up is disjoint from $X$,
and that $\O_X(1)$ extends to some ample invertible sheaf on $X''$.
Thus, it suffices to treat the case that $X$ is projective.

Next, we reduce to the case that   $A$ is disjoint from the set of associated points $\Ass(\O_X)\subset X$:
Let $\shI\subset \O_X$ be a quasicoherent ideal    consisting of  torsion sections supported by a single
point $a\in A$, with  $\length(\shI_a)=1$, such that $\shI_a\simeq\kappa(a)$. In other words,
$\shI_a$ is a 1-dimensional vector subspace inside the socle of $\O_{X,a}$. Let
$X'\subset X$ the corresponding closed subscheme, and $A'=A\cap X'$.
Then there is a module decomposition  
$\O_{X,a}\simeq\O_{X',a}\oplus\shI_a$, which  globalizes to $\O_X\simeq\O_{X'}\oplus\shI$.
Let $\shE'=\shE/\shI\shE$, and suppose there is an integer $t_0$ so that for each $t\geq t_0$,
we have a homomorphism $\shE'\ra\O_{X'}(t)$ such that the cokernel has an infinitesimal neighborhood
of $A'\subset X'$ as support. Using the module decomposition, we obtain
a homomorphism $\shE\ra\O_{X'}(t)\subset\O_{X}(t)$ whose cokernel has an infinitesimal neighborhood
of $A'$ that is strictly larger than $A'$, thus  must contain $A$. Inductively, we are reduced to the
situation that $A$ contains no associated point of $X$.

We now proceed by induction  on the rank $r=\rank(\shE)$. The case $r=1 $ is trivial, because
then the scheme $X$ is zero-dimensional, the closed subscheme $A\subset X$ is also open,
and every locally free sheaf 
is free. Suppose now that $r\geq 2$, and that the assertion is true for $r-1$.

According to the Atiyah--Serre Theorem (see \cite{Kleiman 1969}, Theorem 4.7),
there is an invertible sheaf $\shL\subset \shE$ that is locally a direct summand,
such that $\shF=\shE/\shL$ is locally free of rank $r-1$.  We thus have a short exact sequence
\begin{equation}
\label{serre extension}
0\lra\shL\lra\shE\lra\shF\lra 0.
\end{equation}
Choose an integer $n\geq 0$ so that $\shL^\vee(t)$ and $\O_X(t)$ are globally generated for all $t\geq n$,
and regular global sections
$$
s\in H^0(X,\shL^\vee(n))\quadand s_i\in H^0(X,\O_X(n+i))
$$
for $i=0,\ldots,n-1$ that vanish on $A\subset X$. Let $D,D_i\subset X$ be the corresponding effective Cartier divisors and 
set $X_i=D\cup D_0\cup D_i$, which contain $A$ and have $\dim(X_i)<\dim(X)$, and in particular
$\rank(\shF_{X_i})>\dim(X_i)$.
Let  $\shI\subset \O_X$ be the ideal sheaf of 
$A\subset X$. Choose an integer $n'\geq 0$ so that for all $t\geq n'$, the groups
\begin{gather*}
\Ext^1(\shF,\shI(t)) = H^1(X,\shI\otimes\shF^\vee(t) ), \\
\Ext^1(\shE,\shL(t-3n-i))=H^1(X, \shE^\vee\otimes\shL(t-3n-i)) 
\end{gather*}
vanish for all $i=0,\ldots,n-1$, and that furthermore    there are homomorphisms $\shF_{X_i}\ra\O_{X_i}(t)$
whose cokernels have an infinitesimal neighborhood of $A\subset X_i$ as schematic support. 
The latter can be done by the induction hypothesis applied to the locally free sheaves
$\shF_{X_i}$ on $X_i$ for  $i=0,\ldots,n-1$.

We claim that $t_0=\max(3n,n')$ does the job:
Suppose $t\geq t_0$. Write this integer as $t=n+mn+(n+i)$ for some $m\geq 1$ and some $0\leq i\leq n-1$,
and regard the sections $s,s_i$ as homomorphisms $s:\shL\ra\O_X(n)$ and $s_i:\O_X\ra\O_X(n+i)$.
Their tensor product yields a homomorphism
$$
f=s\otimes s_0^m\otimes s_i:\shL\lra\O_X(t),
$$
whose cokernel is an invertible sheaf on some Cartier divisor. By construction, this Cartier divisor contains $A$,
and is an infinitesimal neighborhood of $X_i$, the latter being defined by $s\otimes s_0\otimes s_i$. It follows that
$f_{X_i}=0$. The short exact sequence (\ref{serre extension}) yields
a long exact sequence
$$
\Hom(\shE,\shI(t))\lra \Hom(\shL,\shI(t))\lra\Ext^1(\shF,\shI(t)),
$$
and the term on the right vanishes. Thus we may extend the homomorphism $f$ to 
a homomorphism $f:\shE\ra\O_X(t)$, denoted by the same letter, which factors
over $\shI(t)$, such that the cokernel is annihilated by $\shI$.  
Now choose a homomorphism $\shF_{X_i}\ra\O_{X_i}(t)$ so that the schematic support of the cokernel is
an infinitesimal neighborhood of $A\subset X_i$, and   let
$$
g:\shE\lra\shF\lra\shF_{X_i}\lra\O_{X_i}(t)
$$
the composite map. 
Arguing in a similar way as above, we may lift this to a map $g:\shE\ra\O_X(t)$ denoted by the same letter.

It remains to check
that the sum $h=f+g:\shE\ra\O_X(t)$ has a cokernel whose schematic support is an infinitesimal neighborhood of $A\subset X$.
On the subsheaf $\shL\subset\shE$, the map $g$ obviously vanishes and $f,h$ coincide.
Since $f$ is surjective outside $X_i$, the same holds for $h$.
On the closed subscheme $X_i\subset X$, the base-change $f_{X_i}$ vanishes, such that $h_{X_i}=g_{X_i}$.
By construction, the cokernel of $g_{X_i}$ has cokernel has schematic support an infinitesimal neighborhood of
$A\subset X_i$, so the same holds for $h_{X_i}$, and thus for $h$, because $A\subset X_i$.
\qed

\begin{remark}
We have used the assumption that the ground field $k$ is infinite 
to apply the Atiyah--Serre Theorem on the existence of 
invertible subsheaves that are locally direct summands, provided that the
rank of the vector bundle exceeds the dimension of the scheme
(compare \cite{Serre 1958}, \cite{Atiyah 1957}, \cite{Kleiman 1969}, and also \cite{Okonek;  Schneider; Spindler 1980}).
We do not know whether this holds true for finite fields.
There might be relations to the Bertini Theorem on smooth hyperplane sections over finite fields
due to Gabber \cite{Gabber 2001} and Poonen \cite{Poonen 2004}.

However, Theorem \ref{infinitely classes} holds true for finite ground fields $k$ if one allows larger ranks:
There is an integer $d\geq 1$ so that there are infinitely many isomorphism
classes of locally free sheaves of rank $r=d\dim(Y)$. Indeed, one takes
a suitable finite field extension $k\subset k'$ so that the construction exists
over $Y'=Y\otimes_kk'$, and obtains the desired locally free sheaves on $Y$ 
as push-forwards of the locally free sheaves on $Y'$.
\end{remark}

\begin{remark}
The proper morphism $f:X\ra Y$ in Theorem \ref{infinitely classes} induces a proper morphism $D\ra f(D)$ whose exceptional set
is finite. Thus $D\ra f(D)$ is finite, and  there is an ample effective divisor $H\subset D$ disjoint from
the exceptional set. Consequently, $f(H)\subset f(D)$ is an effective divisor whose preimage
on $D$ is ample. It follows with \cite{EGA IIIa}, Proposition 2.6.2
that $f(H)$ is ample, such that the proper scheme $f(D)$ is projective.
\end{remark}

\begin{remark}
In Theorem \ref{infinitely classes}, the invertible sheaf $\shL=\O_X(D)$ is relatively semiample over $Y$ by 
the Zariski--Fujita Theorem  \cite{Fujita 1983}. This holds
because the relative base locus, which is contained in $D\cap E$, is finite over $Y$. Replacing $X$ by the relative homogeneous spectrum
of the sheaf of graded $\O_Y$-algebras 
$\shA=f_*(\bigoplus_{t\geq 0}\shL^{\otimes t})$, one thus may as well assume that the exceptional set $E\subset X$
is 1-dimensional.
\end{remark}

\begin{remark}
The arguments in this section hold literally true for   algebraic spaces   or
  complex spaces.
\end{remark}

%===========================================================
\section{Computation of top Chern classes}
\mylabel{Chern Classes}

In this section, we show that the vector bundles $\shE'=\shE_{t,\phi}$ constructed in the proof
for Theorem \ref{infinitely classes} attain infinitely many Chern classes. In fact, their top Chern classes, regarded
as numbers, become arbitrarily large. Naturally, Chern classes for coherent sheaves show up in this computation.
Some care has to be taken for the definition of such Chern classes, because in our situation 
it is not permissible to assume that all coherent sheaves are quotients of locally free sheaves.

For a noetherian scheme $X$, one writes $\Coh(X)$ for the abelian category
of coherent sheaves, and  $K^\circ(X)_\naif$ for the $K$-group of the exact category  $\Vect(X)$
of locally free sheaves of finite rank (compare \cite{SGA 6}, Expose IV, Section 2).
We write $[\shE]\in K^\circ(X)_\naif$ for the class of $\shE\in\Vect(X)$.
The next observation allows us to extend this    to certain $\shF\in\Coh(X)$.

\begin{lemma}
\mylabel{well-defined}
Let $X$ be a noetherian scheme, and $0\ra\shE_1\ra\shE_0\ra\shF\ra 0$
a short exact sequence of coherent sheaves, with $\shE_0$ and $\shE_1$ locally free.
Then the difference $[\shE_0]-[\shE_1]\in K^\circ(X)_\naif$ depends only
on the isomorphism class of $\shF\in\Coh(X)$.
\end{lemma}

\proof
We have to check that the arguments of Borel and Serre (\cite{Borel; Serre 1958}, Section 4), which 
work for exact sequences of arbitrary length but rely on the existence
of global resolutions for \emph{all} coherent sheaves, carry over. 
Suppose $0\ra\shE_1'\ra\shE'_0\ra\shF\ra 0$ is another short exact sequence.
If there is a commutative diagram
$$
\begin{CD}
0 @>>> \shE_1' @>>> \shE'_0 @>>> \shF @>>> 0\\
@.     @Vf_1VV      @Vf_0VV      @VVgV\\
0 @>>> \shE_1  @>>> \shE_0  @>>> \shF@>>> 0\\
\end{CD}
$$
with $f_0,f_1$ surjective and $g$ bijective, the Snake Lemma implies that the induced map $\ker(f_1)\ra\ker(f_0)$
is bijective, and thus $[\shE_0]-[\shE_1]=[\shE'_0]-[\shE'_1]$.

To exploit this in the general case, 
consider the fiber product $\shE''_0=\shE_0\times_\shF\shE'_0$, which can be defined by the exact sequence
$$
0\lra\shE_0''\lra\shE_0\oplus\shE'_0\stackrel{p-p'}{\lra}\shF\lra 0
$$
where $p:\shE_0\ra\shF$ and $p':\shE_0'\ra\shF$ are the canonical maps. The coherent sheaf $\shE''_0$
is already locally free, because the stalks of $\shF$ have projective dimension $\pd(\shF_x)\leq 1$.
One easily sees that the   map $\shE_0''\ra\shF$ given by
$p\circ \pr_1=p'\circ \pr_2 $  is surjective, and its  kernel $\shE''_1$ is thus
also locally free. Moreover, the canonical projection   $\shE_0''\ra\shE_0$ is 
surjective. The Snake Lemma ensures that  the induced map $\shE_1''\ra\shE_1$ is surjective as well.
The preceding paragraph thus gives $[\shE''_0]-[\shE''_1]=[\shE_0]-[\shE_1]$. By symmetry,
we also get $[\shE''_0]-[\shE''_1]=[\shE'_0]-[\shE'_1]$. In turn, the differences in question
depend only on the  isomorphism class of $\shF\in\Coh(X)$.
\qed

\medskip
The following ad hoc terminology will be useful throughout:
Let us call a coherent sheaf $\shF$ \emph{admissible}  
if for each $x\in X$, the projective dimension of the stalk is $\pd(\shF_x)\leq 1$,
and there is surjection $\shE_0\ra\shF$ for some locally free sheaf $\shE_0$ of finite rank. This ensures that
the kernel $\shE_1\subset\shE_0$ is locally free as well.
By the preceding lemma,  the class
$[\shF]=[\shE_0]-[\shE_1]\in K^\circ(X)_\naif$ is   well-defined.

\begin{lemma}
\mylabel{additive}
Let $0\ra\shF'\ra\shF\ra\shF''\ra 0$ be a short exact sequence of admissible coherent
sheaves.   Then $[\shF]=[\shF']+[\shF'']$ in the group $K^\circ(X)_\naif$.
\end{lemma}

\proof
Choose surjections $p:\shE_0'\ra\shF'$ and $q:\shE\ra\shF$ with $\shE'_0$ and $\shE$ locally free of finite rank.
Composing $q$ with the canonical projection $\pr:\shF\ra\shF''$ yields a surjection $\pr\circ q:\shE\ra\shF''$.
We then obtain a commutative diagram
$$
\begin{CD}
0 @>>> \shE'_0 @>\can>> \shE'_0\oplus\shE @>\can>> \shE           @>>> 0\\
@.      @VpVV       @VVp+qV                @VV\pr\circ qV\\
0 @>>> \shF'   @>>> \shF              @>>> \shF''         @>>> 0,
\end{CD}
$$
whose rows are exact and whose vertical maps are surjective. Applying the Snake Lemma, one easily
gets the assertion. 
\qed

\medskip
Given an admissible coherent sheaf $\shF$,  it is thus possible to define the \emph{total  Chern class}
$$
c_\bullet(\shF)=1+ c_1(\shF) + c_2(\shF) + \ldots    = c_\bullet(\shE_0)/c_\bullet(\shE_1)\in A^\bullet(X),
$$
where  $A^\bullet(X)$ is any suitable cohomology theory with Chern classes for locally free sheaves 
satisfying the usual axioms, confer \cite{Borel; Serre 1958}. In light of Lemma \ref{additive},
the \emph{Whitney Sum Formula} $c_\bullet(\shF)=c_\bullet(\shF')c_\bullet(\shF'')$ holds true
for  short exact sequences of admissible sheaves $0\ra\shF'\ra\shF\ra\shF''\ra 0$.

We now examine the following situation: Let $k$ be an infinite ground field, $X$ an irreducible proper scheme
of dimension $n=\dim(X)$, and $\shN$ an invertible sheaf on $X$. Suppose that $D\subset X$ is an effective Cartier divisor,
and $\shL_D$ is an invertible sheaf on $D$ that is the quotient of a locally free sheaf of finite rank
on $X$. Then the same holds for the tensor products $\shN_D\otimes\shL^{\otimes t}_D$ for all $t\geq 0$.
Let us assume that there is a single locally free sheaf $\shA$ of finite rank, having surjections
$\shA\ra \shN_D\otimes\shL^{\otimes t}_D$ for all $t\geq 0$. Denote by $\shE_t\subset\shA$ its kernel,
which is locally free of the same rank, such that we have a short exact sequence
\begin{equation}
\label{single vector bundle}
0\lra\shE_t\lra\shA\lra \shN_D\otimes\shL^{\otimes t}_D\lra0.
\end{equation}
We seek to express the total Chern class $c_\bullet(\shN_D\otimes\shL_D^{\otimes t})\in A^\bullet(X)$,
or rather its inverse, in dependence on $t\geq 0$.
For simplicity,   consider $l$-adic cohomology $A^i(X)=H^{2i}(X,\QQ_l(i))$, where
$l$ denotes a prime number different from the characteristic of the ground field.
We make the identification $A^n(X)=H^{2n}(X,\QQ_l(n))=\QQ_l$ and regard cohomology
classes in top degree as numbers. Moreover, we consider the descending filtration
 $\Fil^jA^\bullet(X)=\bigoplus_{i\geq j} A^i(X)$.

\begin{theorem}
\mylabel{total chern class}
Assumptions as above. Suppose that $\shL_D$ is globally generated.
Then there is a proper birational morphism $f:X'\ra X$ and classes $\alpha_j\in \Fil^{j+1}A^\bullet(X')$, $1\leq j\leq n-1$ such
that 
$$
c_\bullet(f^*(\shN_D\otimes\shL_D^{\otimes t}))^{-1} = 1+\alpha_1t+\ldots+\alpha_{n-1}t^{n-1} 
$$
for all integers $t\geq 0$. Moreover, the coefficient $\alpha_{n-1}\in \Fil^nA^\bullet(X)=A^n(X')$ is  given by $\alpha_{n-1}=(-1)^nc_1^{n-1}(\shL_D)$.
\end{theorem}

\proof
First, we consider the special case that $\shL_D\in\Pic(D)$ is the restriction of some  $\shL\in\Pic(X)$.
The exact sequence
$$
0\lra\shN\otimes\shL^{\otimes t}(-D)\lra\shN\otimes\shL^{\otimes t}\lra\shN_D\otimes\shL_D^{\otimes t}\lra 0
$$
shows that the inverse of the total Chern class for $\shN_D\otimes\shL_D^{\otimes t}$ is 
$$
(1+N+tL-D)/(1+N+tL) = 1 - D\sum_{i=0}^{n-1} (-1)^i(N+tL)^i.
$$
Here $N,L\in A^2(X)$ are the first Chern classes of the invertible sheaves $\shN$ and $\shL$, respectively.
Using the equality of numbers $D\cdot L^{n-1}=c_1^{n-1}(\shL_D)$, the statement already follows with $X'=X$.

Second, consider the special case that $D=D_1\cup D_2$ is the schematic union of two effective Cartier divisors
without common irreducible components. 
Interpreting the intersection number
$c_1^{n-1}(\shL_D)/(n-1)!$ as the the coefficient in degree $n-1$ of  the Hilbert polynomial $\chi(\shL_D^{\otimes t})$,
we deduce $c_1^{n-1}(\shL_{D})=c_1^{n-1}(\shL_{D_1})+c_1^{n-1}(\shL_{D_2})$.
Moreover,  $D_1\cap D_2$ is an effective Cartier divisor
in both $D_1,D_2$, and we have a short exact sequence
$0\ra\O_{D_2}(-D_1)\ra\O_D\ra\O_{D_1}\ra 0$.
Clearly, the restrictions   $\shL_{D_1}, \shL_{D_2}$ are globally generated and admissible. 
Now suppose that our statement is already true for the effective Cartier divisors
$D_1,D_2\subset X$.
Applying the Whitney Sum Formula to the short exact sequence
$$
0\lra\shN'_{D_2}\otimes\shL_{D_2}^{\otimes t} \lra\shN_D\otimes\shL_D^{\otimes t}\lra \shN_{D_1}\otimes\shL_{D_1}^{\otimes t}\lra 0
$$
where $\shN'=\shN(-D_1)$ easily yields the assertion.

We now come to the general situation. We proceed by induction on the \emph{Kodaira--Ithaka dimension}
$k=\kod(\shL_D)$ of the  invertible sheaf $\shL_D\in\Pic(D)$. For globally generated invertible sheaves,
this is  the dimension of the image for the morphism $D\ra\PP^m$ coming from the linear system $H^0(D,\shL_D)$,
where  $m+1=h^0(\shL_D)$. Moreover, it coincides with the \emph{numerical Kodaira--Ithaka dimension},
which is  the largest integer $k\geq 0$ so that the
intersection number
$c_1^k(\shL_D)\cdot V $ is nonzero for some integral closed subscheme $V\subset D$ of dimension $k=\dim(V)$.

In the case $k=0$, we have $\shL_D=\O_D$, and the assertion holds by the first special case. 
Now suppose $k\geq 1$, and that the assertion is true for $k-1$.
Choose a regular global section of $\shL_D$, and let $A\subset D$ be its zero locus, such that $\shL_D=\O_D(A)$.
Let $f:X'\ra X$ be the blowing-up with
center $A\subset X$. Since $f$ is birational, the locally free sheaves
$\shE_t$ and $\shE'_t=f^*(\shE_t)$ have the same top Chern numbers. 
Furthermore, the exact sequence (\ref{single vector bundle}) induces  an exact sequence
$$
0\lra \shE_t'\lra\shA'\lra f^*(\shN_D\otimes\shL^{\otimes t}_D)\lra 0.
$$
where $\shA'=f^*(\shA)$. The latter is indeed exact, because $\O_{X'}$ has no torsion sections supported by 
the effective Cartier divisor $E=f^{-1}(A)$. Thus the coherent sheaf 
$\shF_t=f^*(\shN_D\otimes\shL^{\otimes t}_D)$ is admissible.

Consider the effective Cartier divisor $D'=f^{-1}(D)$.
The universal property for blowing-ups gives a partial section
$\sigma:D\ra X'$ for the structure morphism $f:X'\ra X$.
We thus obtain  a short exact sequence
$$
0\lra\shF_t(-\sigma(D))|E \lra \shF_t \lra \shF_t|{\sigma(D)}\lra 0,
$$
where the term on the right is invertible on the effective Cartier divisor $\sigma(D)\subset X'$,
and the term on the left is invertible on the effective Cartier divisor $E\subset X'$.
This follows from Lemma \ref{blowing-up} below.
We now define another locally free sheaf $\tilde{\shE}'_t$ as the kernel
of the composite surjection $\shA'\ra\shF_t\ra\shF_t|\sigma(D)$, such that we have a commutative
diagram
$$
\begin{CD}
0 @>>> \shE'_t         @>>> \shA'@>>> \shF_t  @>>> 0\\
@.     @VVV                 @VV\id V      @VV\can V      \\
0 @>>> \tilde{\shE}'_t @>>> \shA'@>>> \shF_t|{\sigma(D)} @>>> 0.
\end{CD}
$$
By the Snake Lemma, the vertical map on the left is injective,
with cokernel 
$$
\shF_t(-\sigma(D))|E = f^*(\shN_A \otimes\O_A(tA))(-\sigma(D)|E).
$$
In turn, this sheaf is admissible.
The Cartier divisor $D'=f^{-1}(D)\subset X'$ is the union of the two effective Cartier divisors $E,\sigma(D)$, which have
no common irreducible component. Consider the globally generated 
invertible sheaf $\shL_{D'}=f^*(\shL_D)=f^*(\O_D(A))$ on $D'$.
Its restriction to $\sigma(D)$ is nothing but $\O_{\sigma(D)}(E)$, which is the restriction of the invertible sheaf
$\O_{X'}(E)$. Thus the assertion hold for $\shL_{D'}|\sigma(D)$ by the first special case. 
Moreover, the restriction   $\shL_{D'}|E$ equals $f^*(\O_A(A))$. This sheaf has  Kodaira--Ithaka dimension 
smaller than that of  $\shL_D=\O_D(A)$, by its interpretation via intersection numbers. 
We thus may apply the induction hypothesis to $\shL_{D'}|E$.
Using the second special case, we infer the assertion for $\shL_{D'}$.
\qed

\medskip 
In the preceding proof, we have used the following fact:

\begin{lemma}
\mylabel{blowing-up}
Let $X$ be a noetherian scheme, $D\subset X$ an effective Cartier divisor, $A\subset D$ an effective Cartier divisor,
and $f:X'\ra X$ the blowing-up with center $A$. Let  $\sigma:D\ra X'$ be the canonical partial section,
$E=f^{-1}(A)$ the exceptional divisor, and  $D'=f^{-1}(D)$. Then $E,D', \sigma(D)\subset X'$ are effective Cartier divisors,
the subschemes $E,\sigma(D)\subset X'$ have no common irreducible components, their schematic union
is  $D'$, and there is a short exact  sequence
\begin{equation}
\label{subscheme sequence}
0\lra \O_E(-\sigma(D))\lra\O_{D'}\lra\O_{\sigma(D)}\lra 0.
\end{equation}
\end{lemma}

\proof
We   give a sketch and leave some  details to the reader.
The problem is local, so it suffices to treat the case that $X=\Spec(R)$ is affine, and that
there is  a regular sequence $f,g\in R$ such that $D\subset X$ is defined by the ideal $fR$,
and $A\subset X$ is defined by the ideal $I=(f,g)$. 

Consider the Rees ring $S=\bigoplus I^n$,
such that $X'=\Proj(S)$. The closed subscheme $D'\subset X'$ can be identified with the homogeneous spectrum
of $S\otimes_R(R/fR)=\bigoplus I^n/fI^n$, the exceptional divisor $E\subset X'$ with that of
of $S\otimes_R(R/I)=\bigoplus I^n/I^{n+1}$, and the section $\sigma(D)\subset X'$ with that of the
graded  $R/fR$-algebra
$\bigoplus J^n$, where $J=I\cdot R/fR$ is the induced invertible ideal. The homogeneous components  of the latter can be rewritten
as $J^n=(I^n+fR)/fR=I^n/(fR\cap I^n)$.

Thus, in order to verify $D'=E\cup\sigma(D)$, it suffices to check that for each degree $n\geq 1$,
the sequence
\begin{equation}
\label{homogeneous components}
0\lra I^n/fI^n \lra I^n/I^{n+1} \times I^n/(fR\cap I^n) \lra  I^n/(I^{n+1}+(fR\cap I^n))\lra 0
\end{equation}
is exact. Consider first the special case that $R=\ZZ[f,g]$, where $f,g$ are indeterminates.
One easily sees that each term in (\ref{homogeneous components}) is a free $\ZZ$-module, 
with basis given by certain monomials in $f$ and $g$, from which one easily infers exactness. 

Moreover, it follows from  \cite{Eagon; Hochster 1974}, Theorem 2 that 
if $M$ is a module over $\ZZ[f,g]$ such that $f,g$ is a regular sequence on $M$,
then $\Tor_p(\ideala/\idealb,M)=0$ for all $p>0$ and all monomial ideals $\ideala,\idealb\subset \ZZ[f,g]$
occurring according in the terms of (\ref{homogeneous components}), for example $\ideala=I^n$ and $\idealb=fI^n$.
Using the long exact sequences for Tor, we infer that the exactness of (\ref{homogeneous components})
for the ring $\ZZ[f,g]$ implies the exactness for any local rings $R$ with regular sequence $f,g\in R$.
This indeed shows that $D'=E\cup\sigma(D)$ holds in general.

By the universal property of blowing-ups, the closed subscheme $E\subset X'$ is an effective
Cartier divisor. Since $f$ is regular on $R$, the same holds for the polynomial ring $R[T]$
and the subalgebra $S\subset R[T]$, such that $D'\subset X'$ is an effective Cartier divisor.
Localizing at the generic points of $D\subset X$, one easily sees that $\sigma(D)$ and
$E$ have no common irreducible component. Using that there are no associated points on $X$ and $X'$
contained in the Cartier divisors $D$ and $E$, respectively, we infer that $\sigma(D) = D'-E$
is an effective Cartier divisor. 
Finally, the Snake Lemma applied to the diagram
$$
\begin{CD}
0 @>>> \O_{X'}(-E-\sigma(D)) @>>> \O_{X'} @>>> \O_{D'} @>>> 0\\
@.     @VVV                       @VVV         @VVV\\
0 @>>> \O_{X'}(-\sigma(D))   @>>> \O_{X'} @>>> \O_{\sigma(D)} @>>> 0
\end{CD}
$$
yields the short exact sequence (\ref{subscheme sequence}).
\qed

\medskip
As an application, we now can compute the top Chern class for the locally free sheaves
constructed in the course of the proof for Theorem \ref{infinitely classes}. We work in the following set-up:
Let  $X$ be a  proper irreducible scheme of dimension $n=\dim(X)$.
Suppose  $D\subset X$ is 
an effective Cartier divisor, and $\shN$ an invertible sheaf on $X$, and $\shL_D$ be
a globally generated invertible sheaf on $D$.
Suppose we have a locally free sheaf $\shA$ of finite rank, and for each $t\geq 0$
some surjection $\shA\ra\shN_D\otimes\shL_D^{\otimes t}$.
Define the locally free sheaf $\shE_t$ by the short exact sequence
\begin{equation}
\label{defining sequence}
0\lra\shE_t^\vee\lra\shA\lra \shN_D\otimes\shL_D^{\otimes t}\lra 0,
\end{equation}
and regard its top Chern class $c_n(\shE_t)\in A^n(X)=H^{2n}(X,\QQ_l(n))=\QQ_l$ as a number.

\begin{proposition}
\mylabel{chern number}
Assumptions as above.  
Then  there are  $\beta_0,\ldots,\beta_{n-1}\in\QQ_l$ with
$$
c_n(\shE_t)=\beta_{n-1}t^{n-1}+\beta_{n-2}t^{n-2}+\ldots+\beta_0,
$$ 
for all $t\geq 0$, and the coefficient in 
degree $n-1$ is given by  $\beta_{n-1}= c_1^{n-1}(\shL_D)$.
\end{proposition}

\proof
Applying the Whitney Sum Formula
to the short exact sequence (\ref{defining sequence}) and using Theorem \ref{total chern class}, we see that
$$
c_\bullet(\shE_t^\vee)=(1+c_1(\shA)+\ldots+c_n(\shA)) \cdot (1+\alpha_1t+\ldots+\alpha_{n-1}t^{n-1})
$$
for certain cohomology classes $\alpha_j\in\Fil^{j+1}A^\bullet(X)$, with $\alpha_{n-1}=(-1)^nc_1^{n-1}(\shL_D)$.
Strictly speaking, the classes $\alpha_j$ lie on $X'$ for some proper birational morphism $X'\ra X$,
but this can be neglected because the canonical map $A^n(X)\ra A^n(X')$ is bijective.

We conclude that $c_n(\shE_t^\vee)$ is a polynomial function of degree $\leq n-1$ in $t\geq 0$. Its coefficient in
degree $n-1$ is $(-1)^nc_1^{n-1}(\shL_D)$, because $c_i(\shA)\cdot\alpha_{n-1}\in\Fil^{n+i}A^\bullet(X)=0$ for all $i\geq 1$.
The statement follows from the general fact that $c_i(\shE^\vee)=(-1)^ic_i(\shE)$ for any locally free sheaf $\shE$
of finite rank.
\qed

\medskip
Combining this computation with the proof for Theorem \ref{infinitely classes}, we obtain the following,
which is the main result of this paper:

\begin{theorem}
\mylabel{main result}
Let $Y$ be a proper scheme.
Suppose there is a proper birational morphism $X\ra Y$ and a
Cartier divisor $D\subset X$ that intersects the exceptional locus  
in a finite set, and that the proper scheme $D$ is projective. 
Then there are locally free sheaves $\shE$ of rank $n=\dim(Y)$ on $Y$ with Chern number $c_n(\shE)$ arbitrarily large.
\end{theorem}

Over the field of complex numbers $k=\CC$, we
see that there are infinitely many algebraic vector bundles of rank $n=\dim(Y)$
that are non-isomorphic as topological vector bundles.
Of course, in this situation one could  use singular cohomology  $A^i(X)=H^{2i}(X^\an,\ZZ)$ 
of the associated complex space $X^\an$ rather then
$l$-adic cohomology $A^i(X)=H^{2i}(X,\QQ_l(i))$.

%===========================================================
\section{Toric varieties}
\mylabel{Toric Varieties}

In this section we   study toric varieties and formulate a condition ensuring that a given toric prime divisor becomes
$\QQ$-Cartier on a small modification. If furthermore the toric divisor is projective,   we   conclude
with Theorem \ref{main result} that there are non-trivial vector bundles. It turns out that
this condition automatically holds in dimension $n=3$.
For general facts on toric varieties, we refer to \cite{Cox; Little; Schenck 2011} and
\cite {Kempf et al. 1973}.

Fix an infinite ground field $k$. As customary, we denote by $T = \mathbb{G}_m^n$ the standard torus, $M=\ZZ^{\oplus n}$
its character group and $N = \Hom_\ZZ(M, \ZZ)$ the dual group of $1$-parameter subgroups. Moreover,
  set $N_\RR = N \otimes_\ZZ \RR$ and $M_\RR = M \otimes_\ZZ \RR$.
Any \emph{strictly convex rational polyhedral cone} $\sigma \subset N_\RR $ is of the form
$\sum_{i = 1}^t \RR_{\geq 0} v_i$, where $v_1, \dots, v_t\in N\subset N_\RR$ are lattice vectors, such that
$\sigma$ does not contain non-trivial linear subspaces. Its dual cone is given by $\check{\sigma} = \{m \in M_\RR \mid
n(m) \geq 0$ for all $n \in \sigma\}$. The affine toric variety $U_\sigma$ associated to $\sigma$ is
the spectrum of the monoid ring $k[\check{\sigma} \cap M]$.

The \emph{linear span} of $\sigma$ is the linear subspace of $N_\RR$ generated by $\sigma$. 
Its dimension is also called the dimension of the cone $\sigma$.  Inside this linear
span, we can distinguish between the interior and the boundary of $\sigma$. We call the former the
{\em relative interior} of $\sigma$.
A {\em face} of $\sigma$ is either $\sigma$ itself or given by an intersection $\sigma \cap H$, where
$H\subset H_\RR$ is a hyperplane  disjoint from the relative interior of $\sigma$. A proper face (i.e.\
a face $\neq \sigma$) is again a strictly convex rational polyhedral cone  contained in $H$.
The set of faces of $\sigma$ is closed under intersection and partially ordered, where we write
$\tau \preceq \eta$ if and only if $\tau \subseteq \eta$. For any integer $l\geq 0$ we denote by $\sigma(l)$ the set of
$l$-dimensional faces of $\sigma$.
The most important faces   are  the {\em rays} $\rho\in\sigma(1)$ and the {\em facets}
$\eta\in\sigma(d - 1)$, where $d = \dim \sigma$.
Note that $\sigma = \sum_{\rho \in \sigma(1)} \rho$.

From now on, we assume that $\dim \sigma = n$. Given  a ray $\rho \in \sigma(1)$, we shall formulate certain conditions on
the pair $(\sigma, \rho)$. For this, we use the 
{\em beneath-and-beyond method} from convex geometry 
(see \cite{Edelsbrunner 1987}, \S 8.4). In its original form, it
deals with   polytopes rather than cones. However, we can always choose
an affine hyperplane $H$ which passes through the interior of $\sigma$ such that the cross-section $P =
\sigma \cap H$ is a compact polytope and $\sigma$ coincides with the $\RR_{\geq 0}$-linear span of $P$.
Moreover, there is a one-to-one correspondence between the non-zero faces of $\sigma$ and those of $P$,
where the former are the $\RR_{\geq 0}$-linear spans of the latter. This way, the corresponding terminology 
 in \cite{Edelsbrunner 1987} straightforwardly carries over to our setting.

Consider a facet $\eta = \sigma \cap H_\eta$.
Then the hyperplane $H_\eta$, which  is unique for facets, separates $N_\RR$ into two half spaces, one of which contains the interior of $\sigma$.
We say that $x\in N_\RR$ is {\em beneath} $H_\eta$ if it is contained in the same half space that
contains the relative interior of $\sigma$; if it is contained in the opposite half space, we call $x$ {\em beyond}
$H_\eta$.  
Now set 
$$
\sigma' = \sum_{\rho' \in \sigma(1) \smallsetminus \{\rho\}} \rho'.
$$
One can distinguish
two cases: First, $\dim \sigma' = n - 1$ and therefore $\sigma'\subset\sigma$ is a proper face. Second,
$\dim \sigma' = n$, i.e.\ $\sigma'$ is a cone contained in $\sigma$ and, in general, will share some of its
faces with $\sigma$. The beneath-and-beyond method is an inductive procedure to describe the faces of $\sigma$
in terms of  the faces of $\sigma'$ and the relative position of $\rho$ with respect to the hyperplanes of the facets of
$\sigma'$. Here, we are interested   in the following special case:

\begin{definition}
Let $\rho$, $\sigma' \subset \sigma$ be as before.
We say that $\sigma$ is a {\em pyramidal extension} of $\sigma'$ by $\rho$, if one of the following holds:
\begin{enumerate}
\item $\dim \sigma' = n - 1$.
\item $\dim \sigma' = n$ and there exists precisely one facet $\eta\subset\sigma'$ such that the relative
interior of $\rho$ is beyond $H_\eta$, and beneath every other facet of $\sigma'$.
\end{enumerate}
In the first case, we set $\sigma'' = \sigma = \sigma' + \rho$. In the second case we write $\sigma'' = \eta + \rho$.
\end{definition}

Note that in both cases, $\sigma''$ is an $n$-dimensional cone, which in \cite{Edelsbrunner 1987} is called 
\emph{pyramidal update}   of $\sigma'$ or $\eta$, respectively.

\begin{example}\label{pyramid}
Let $\tau$ be a $(n - 1)$-dimensional cone and $\kappa = \RR_{\geq 0} x$ for some $x \in N_\RR \smallsetminus H_\tau$.
Then $\sigma = \tau + \kappa$ is a pyramidal extension with $\sigma' = \tau$, $\sigma'' = \sigma$ and $\rho
= \kappa$.
\end{example}

\begin{example}\label{simplex}
Recall that an $n$-dimensional cone is called {\em simplicial} if it is generated by $n$ rays. In this case, for any
$\rho \in \sigma(1)$, the cone $\sigma'$ as above is $(n - 1)$-dimensional
and therefore $\sigma$ is a pyramidal extension of $\sigma'$ by $\rho$.
\end{example}

\begin{example}\label{threedimensionalcones}
Let $\sigma$ be a non-simplicial $3$-dimensional cone and $\rho$ be any element of $\sigma(1)$. Then
$\rho = \tau_1 \cap \tau_2$ for two facets $\tau_1, \tau_2 \in \sigma(2)$ and there are $\kappa_1, \kappa_2 \in
\sigma(1)$ such that $\tau_1 = \rho + \kappa_1$ and $\tau_2 = \rho + \kappa_2$. Then it is elementary to see
that $\sigma$ is a pyramidal extension of the $n$-dimensional cone $\sigma' = \sum_{\rho' \in \sigma(1) \smallsetminus\{\rho\}} 
\rho'$ with respect to the facet $\eta =
\kappa_1 + \kappa_2$. Together with Example \ref{simplex}, this shows  that every $3$-dimensional cone is
a pyramidal extension.
\end{example}

 We now paraphrase from  \cite{Edelsbrunner 1987}, Lemmas 8.6 and 8.7  
the classification of faces of $\sigma$ and $\sigma''$ in terms of those of $\sigma'$
and $\eta$ for   pyramidal extensions.

\begin{lemma}\label{pyramidalfaces}
Let $\rho, \sigma', \sigma'' \subset \sigma$ be as before.
\begin{enumerate} 
\item\label{pyramidalfacesi} If $\dim \sigma' = n - 1$, then every face of $\sigma'$ is a face of $\sigma''$.
Moreover, any cone of the
form $\tau + \rho$ with $\tau \preceq \sigma'$ is a face of $\sigma''$ and $\sigma''$ has no other faces.
\item\label{pyramidalfacesii}
If $\dim \sigma' = n$, then every face of $\eta$ is a face of $\sigma''$. Moreover, any cone of the
form $\tau + \rho$ with $\tau \preceq \eta$ is a face of $\sigma''$ and $\sigma''$ has no other faces.
\item\label{pyramidalfacesiii}
 If $\dim \sigma' = n$, then every proper face of $\sigma'$ is a face of $\sigma$ if it does not coincide with
$\eta$. For every proper face $\tau$ of $\eta$, the cone $\tau + \rho$ is a proper face of $\sigma$. There are no
other proper faces of $\sigma$.
\end{enumerate}
\end{lemma}

Note that the characterization (i) will be convenient later on though it is somewhat redundant,
because $\sigma'' = \sigma$ holds.
Recall that a collection $\Delta$ of cones is  called a {\em fan} if it is closed
under taking intersections and faces.

\begin{proposition}\label{subdivisionfan}
If $\dim \sigma' = n$, then $\sigma' \cap \sigma'' = \eta$, and $\sigma' \cup \sigma'' = \sigma$, and the faces of
$\sigma'$ and $\sigma''$ form a fan.
\end{proposition}

To any toric variety $Y$ there is associated   a fan $\Delta$ that encodes the orbit structure of the $T$-action
on $Y$. More precisely, every $\sigma \in \Delta$ corresponds to an open affine subset $U_\sigma =
\Spec K[\check{\sigma} \cap M]$ which, in turn, contains a unique minimal orbit $\orb(\sigma)$. This correspondence is
compatible with inclusions, i.e.\ $\tau \preceq \sigma$ if and only if $U_\tau \subseteq U_\sigma$ if and only if $\orb(\sigma)$
is contained in the closure of $\orb(\tau)$ in $Y$. An orbit closure $V(\sigma) = \orb(\sigma)$
by itself is again a toric variety with respect to the action of the torus $T / T_\sigma$, where $T_\sigma \subset T$
denotes the stabilizer subgroup of $\orb(\sigma)$. Its orbit structure can be described with help of the {\em star}
$\Star(\sigma) = \{\tau \in \Delta \mid \sigma \preceq \tau\}$. All $\tau
\in \Star(\sigma)$ have $\sigma$ as a common face and the sets $\bar{\tau} = (\tau + \langle \sigma \rangle_\RR) /
 \langle \sigma \rangle_\RR$ form a fan $\bar{\Delta}_\sigma$ in $N_\RR / \langle \sigma \rangle_\RR$ which encodes
the toric structure of $V(\sigma)$.

Again, we are only interested in two special situations. The first is the case where $\tau$ is $(n - 1)$-dimensional.
Then $\dim V(\tau) = 1$ and there are only three possibilities for what $V(\tau)$ can look like, depending on whether
$\tau$ is contained in none, one, or two $n$-dimensional cones which correspond to $V(\tau) \simeq \mathbb{A}^1
\smallsetminus \{0\}$, $V(\tau) \simeq \mathbb{A}^1$, and $V(\tau) \simeq \PP^1$, respectively.

The second case is where
$\rho$ is a ray in $\Delta$ such that $V(\rho) = D_\rho$ is a {\em toric Weil divisor}.
In what follows, we write $\Delta(1)$ for the collection of rays in the fan $\Delta$, as customary.

\begin{lemma}\label{pyramidalcartier}
Let $\rho \in \Delta(1)$ be such that for every $n$-dimensional cone $\sigma \in \Star(\rho)$, we have
$\dim \sigma' = n - 1$. Then $D_\rho \subset Y$ is $\QQ$-Cartier.
\end{lemma}

\begin{proof}
We have to find an integer $c\neq 0$ so that $cD_\rho\subset Y$ is Cartier.
Any toric Weil divisor of the form $D = \sum_{\kappa \in \Delta(1)} a_\kappa D_\kappa$ with $a_\kappa \in \ZZ$,
and $D$ is Cartier if and only if for every maximal cone $\sigma \in \Delta$ there exists $m_\sigma \in M$ such that
$m(l_\kappa) = -a_\kappa$ for every $\kappa \in \sigma(1)$, where $l_\kappa \in \N$ denotes the \emph{primitive generator}
of $\kappa$. For our divisor $D_\rho$, we have $a_\kappa = 0$ for $\kappa \neq \rho$ and $a_\rho = 1$. Therefore,
for any $\sigma \in \Delta$ which does not contain $\rho$, we can choose $m_\sigma = 0$. If $\rho \in \sigma(1)$
then, because the rays of $\sigma'$ lie in a hyperplane in $M_\RR$, we can choose $m'_\sigma \in M_\RR$ so that
$m'_\sigma(l_\kappa) = 0$ for every $\kappa \in \sigma'(1)$ and $m'_\sigma(l_\rho) = 1$. Then for a suitable multiple
$c$, we have $m_\sigma = c \cdot m'_\sigma \in M$ for all $\sigma \in \Star(\rho)$ and hence the $m_\sigma$ describe
the Cartier divisor $c D_\rho$.
\end{proof}

Under the assumptions of Proposition \ref{subdivisionfan}, the cones $\sigma'$ and $\sigma''$ generate a fan
which is supported on
$\sigma$. It arises from the fan generated by $\sigma$ by removing the cone $\sigma$ and
introducing three new cones  $\sigma',\sigma'',\sigma' \cap \sigma''$. 
This implies that $\sigma' \cup \sigma''$ extends to a refinement $\Delta'$ of the
fan $\Delta$ such that the associated toric morphism $Y' \rightarrow Y$ corresponds to the contraction of a
toric subvariety $V(\eta) \simeq \mathbb{P}^1$. Moreover, if $D_\rho'$ denotes the strict transform of $D_\rho$
on $X$, then $D'_\rho \cap V(\eta)$ consists of precisely one point.

\begin{definition}
We say that $\rho \in \Delta(1)$ is in {\em Egyptian position} if every $n$-dimensional cone $\sigma \in
\Star(\rho)$ is a pyramidal extension of $\sigma'$ by $\rho$.
\end{definition}

If $\Delta$ contains a ray $\rho$ in Egyptian position, then for every $\sigma \in \Star(\rho)$ with $\dim \sigma'
= n$, we can consider the modification of $X \rightarrow Y$ corresponding to inserting the extra facets $\sigma'
\cap \sigma''$ for every such $\sigma$. By our discussion above, the exceptional locus $E\subset X$  
if a disjoint union of copies of  $\mathbb{P}^1$, such that $\dim E \cap D_\rho' = 0$. Summing up:

\begin{proposition}\label{Egyptiandivisor}
Let $Y$ be an $n$-dimensional toric variety associated to a fan $\Delta$ in $N$ and $\rho \in \Delta(1)$ in
Egyptian position. For the corresponding toric modification $f : X \ra Y$ with exceptional set $E\subset X$,
denote $D'_\rho \subset X$ the strict transform of the toric prime divisor $D_\rho \subset Y$ associated to
$\rho$. Then $E$ is a curve, $D_\rho' \cap E$ is finite, and the induced morphism $D'_\rho \ra D_\rho$ is an
isomorphism.
\end{proposition}

In general,   multiples of a divisor that is quasiprojective are not necessarily quasiprojective.
The following shows that this problem does not occur for toric prime divisors.

\begin{proposition}\label{multiples}
Let $Y$ be an $n$-dimensional toric variety with fan $\Delta$ and $D = D_\rho$ a  toric prime divisor for
some $\rho \in \Delta(1)$. If the scheme $D$ is quasiprojective, the
same holds for    $cD$  for all integers $c > 0$.
\end{proposition}

\begin{proof}
The divisor $D\subset Y$ is contained in the open subset $U = \bigcup_{\sigma \in \Star(\rho)} U_\sigma$. It is enough
to show that $U$ admits an ample invertible sheaf. By construction, there is a toric morphism
$\pi: U \ra D_\rho$ which is
induced by a map of fans $\pi': \Delta_\rho \ra \bar{\Delta}_\rho$, where $\Delta_\rho$ denotes the fan
generated by $\Star(\rho)$ and $\bar{\Delta}_\rho$ is the fan describing
$D_\rho$ as a toric variety. The map $\pi'$ is induced by the projection
$N_\RR \ra N_\RR / \langle\rho\rangle_\RR $. In particular, there is a one-to-one correspondence between maximal cones in $\Delta_\rho$ and
maximal cones in $\bar{\Delta}_\rho$, given by $\sigma \mapsto \pi'(\sigma) = \bar{\sigma}$.
Consequently, for every maximal cone $\bar{\sigma}$ and the corresponding open toric variety $U_{\bar{\sigma}}$,
we have $\pi^{-1}(U_{\bar{\sigma}}) = U_\sigma$. Hence the morphism
$\pi$ is affine. Therefore, by \cite{EGA II}, Proposition 5.1.6, the structure sheaf $\mathcal{O}_U$ is $\pi$-ample
and with \cite{EGA II}, Proposition 4.6.13 we conclude that $U$ admits an ample invertible sheaf.
\end{proof}

\begin{remark}
The fact that $U$ as   in the preceding proof  is quasiprojective if and only if
$D_\rho$ is, has a   nice interpretation in terms of the toric combinatorics. Recall that a very ample toric divisor
$D = \sum_{\xi \in \Delta_\rho(1)} c_\xi D_\xi$ on $U$ corresponds to an integral polyhedron
$P_D = \{m \in M \mid l_\xi(m) \geq -c_\xi\}$ whose face lattice is dual to that of $\Delta_\rho$.
The restriction of $\mathcal{O}(D)$ to $D_\rho$ corresponds to a divisor $D'$ on $D_\rho$ with associated polyhedron
$P_{D'} \subset \rho^\bot \cap M$. Up to rational equivalence, we can always assume that $c_\rho = 0$ and then
we can identify in a natural way $P_{D'}$ with $\rho^\bot \cap P_D$, which is the face of $P_D$ orthogonal to $\rho$.

Conversely, every ray $\bar{\tau}$ in $\bar{\Delta}_\rho$ is the image of a two-dimensional cone $\tau$ in
$\Star(\rho)$, which in turn is generated by $\rho$ and another ray $\xi \in \Delta_\rho(1)$. However, the
generator of $\xi$ might not map to a generator of $\bar{\tau}$. So, if $D' = \sum_{\tau \in \bar{\Delta}(1)}
c_{\bar{\tau}} D_{\bar{\tau}}$ is an ample toric divisor on $D_\rho$ with associated polyhedron $P_{D'} \subset
\rho^\bot \cap M$, the naturally defined polyhedron $P_D$ with $P_{D'} = P_D \cap \rho^\bot$ might only represent
an ample $\QQ$-Cartier divisor which can be made integral by passing form $D'$ to an appropriate multiple.
\end{remark}

We now come to the main result of this section:

\begin{theorem}\label{manybundles}
Let $Y$ be  a proper toric variety  with associated fan $\Delta$. Suppose there is a ray $\rho \in \Delta(1)$ in
Egyptian position such that the corresponding toric prime divisor $D_\rho$ is projective. Then there are locally
free sheaves $\shE$ on $Y$ of rank $n=\dim(Y)$ with Chern number $c_n(\shE)$ arbitrarily large.
\end{theorem}

\begin{proof}
By Proposition \ref{Egyptiandivisor}, we can subdivide $\Delta$ so that the corresponding modification
$X \rightarrow Y$ has $1$-dimensional exceptional set, whose intersection  with the strict transform of $D_\rho$ on
$X$ is zero-dimensional. By Lemma \ref{pyramidalcartier}, $D_\rho$ is $\QQ$-Cartier and therefore there exists
a multiple $c > 0$ such that $c D_\rho$ is Cartier. Moreover,   $c D_\rho$ remains projective by
Proposition \ref{multiples}.
Hence, we can apply Theorem \ref{main result}, which proves the assertion.
\end{proof}

\begin{corollary}\label{manybundles3d}
On every $3$-dimensional proper toric variety $Y$ there are  locally free sheaves $\shE$ of rank $3$ with arbitrarily large 
Chern number $c_3(\shE)$.
\end{corollary}

\begin{proof}
We saw in Example \ref{threedimensionalcones} that in a $3$-dimensional fan every ray is in Egyptian position.
Moreover, every toric prime divisor is a toric surface and therefore projective, so the  Theorem   applies.
\end{proof}

%===========================================================
\section{Examples with trivial Picard group}
\mylabel{Examples trivial}

In this section we will construct in any dimension $n \geq 3$ an explicit family  of toric varieties with
trivial Picard group which admit a ray in Egyptian position. Let $e_1, \dots, e_n$ be the standard basis of
$N = \ZZ^n$ and $u > 0$ an integer. Consider the following $2n + 2$ primitive vectors:
\begin{align*}
e & = e_n, \quad f_i = e_i \text{ for } 1 \leq i < n, \quad f_n = -\sum_{i = 1}^{n - 1} e_i,\\
h & = -e,  \quad g_i = h - f_i \text{ for } 1 \leq i < n, \quad g_n = uh - f_n.
\end{align*}
With these vectors, we define the following $\binom{n + 1}{2}$  cones of dimension $n$:
\begin{align*}
\sigma_i & = \langle e, g_i, f_k \mid  k \neq i \rangle_{\RR_{\geq 0}} \text{ for every } 1 \leq i \leq n,\\
\sigma_{ij} & = \langle h, g_i, g_j, f_k \mid k \neq i, j \rangle_{\RR_{\geq 0}} \text{ for every pair } 1 \leq i \neq j \leq n.
\end{align*}

Let us now show that these cones generate a fan. We start by analyzing their face structure.
Every cone has precisely $n + 1$ generators and it is easy to see that the generators of a cone form a \emph{circuit}, i.e.\
a minimally linearly dependent set of lattice vectors. In particular, the generators of the $\sigma_i$ satisfy the
following relations:
$$
e + g_i = \sum_{j \neq i} f_j \text{ for } 1 \leq i < n, \quad \text{ and }\quad ue + g_n = \sum_{j = 1}^{n - 1} f_j,
$$
and for the generators of the $\sigma_{ij}$ we get:
$$
g_i + g_j = 
\begin{cases}
2h + \sum_{k \neq i, j} f_k & \text{ if } i,j\neq n; \\
(u+1)h + \sum_{k \neq i, j} f_k & \text{ else}.
\end{cases}
$$
The face structures of the cones $\sigma_i, \sigma_{ij}$ can easily be read-off from these relations. In particular,
every facet is simplicial (for details we refer to \cite{GKZ}, \S 7). The $2n - 2$ facets of $\sigma_i$ are:
\begin{align*}
\langle e, f_j \mid j \neq i, k \rangle_{\RR_{\geq 0}} \quad \text{ and } \quad
\langle g_i, f_j \mid j \neq i, k \rangle_{\RR_{\geq 0}} \text{ for every } 1 \leq k \neq i \leq n,
\end{align*}
and the $2n - 2$ facets of $\sigma_{ij}$ are:
\begin{align*}
&\langle h, g_i, f_l \mid j \neq i, j, k \rangle_{\RR_{\geq 0}} \quad
\langle h, g_j, f_l \mid j \neq i, j, k \rangle_{\RR_{\geq 0}} \text{ for } 1 \leq k \neq i, j \leq n\\
\text{ and } & \langle g_i, f_k \mid k \neq i, j \rangle_{\RR_{\geq 0}}, \quad
\langle g_j, f_k \mid k \neq i, j \rangle_{\RR_{\geq 0}}.
\end{align*}
We have the following intersections of codimension one among the $\sigma_i, \sigma_{ij}$:
\begin{align*}
\sigma_i \cap \sigma_j & = \langle e, f_k \text{ with } k \neq i, j \rangle_{\RR_{\geq 0}} \text{ for } i \neq j,\\
\sigma_{ik} \cap \sigma_{jk} & = \langle h, g_k, f_l, l \neq i, j, k \rangle_{\RR_{\geq 0}}  \text{ for } i \neq j,\\
\sigma_i \cap \sigma_{ij} & = \langle g_i, f_k, k \neq i, j \rangle_{\RR_{\geq 0}}  \text{ for } i \neq j.
\end{align*}
The remaining intersections are all of codimension three:
\begin{align*}
\sigma_{ij} \cap \sigma_{pq} & = \langle h, f_k, k \neq i, j, p, q \rangle_{\RR_{\geq 0}} \text{ if } \{i, j\} \cap \{p, q\}
= \emptyset,\\
\sigma_i \cap \sigma_{jk} & = \langle f_l, l \neq i, j, k \rangle_{\RR_{\geq 0}} \text{ if } i \neq j, k.
\end{align*}
So we see that any two cones intersect in a proper face and every facet is the intersection of two maximal
cones. It follows that the cones $\sigma_i$, $\sigma_{ij}$ generate a complete fan $\Delta_u$. We denote $Y_u$
the corresponding proper toric variety.

\begin{proposition}\label{trivialpic}
The toric variety $Y_1$ is projective,  whereas $\Pic(Y_u) = 0$ for $u > 1$.
\end{proposition}

\proof
In general, on an $n$-dimensional proper toric variety $Y$ with fan $\Delta$, a toric Weil divisor
$D = \sum_{\rho \in \Delta(1)} c_\rho D_\rho$ is Cartier if and only if there exist a collection of characters
$(m_\sigma)_{\sigma \in \Delta(n)}$ such that $c_\rho = -m_\sigma(l_\rho)$ for every $\sigma \in \Delta$ with
$\rho \in \sigma(1)$. Here we identify $M$ with the dual of $N$ and  write $m_\sigma(l_\rho)$ for the
evaluation  of $m_\sigma$ at the primitive vector $l_\rho$ generating the ray $\rho$.
Note that if $(m_\sigma)_{\sigma \in \Delta(n)}$ corresponds to a toric Cartier divisor,
then so does $(m_\sigma + m)_{\sigma \in \Delta(n)}$ for any $m \in M$; this corresponds to a change of
linearization of the line bundle $\mathcal{O}_X(D)$ by a global twist with $m$.

In our situation, we denote the toric prime divisors by $D_e, D_{f_i}, D_{g_i}, D_h$ and consider a family of characters
$m_i, m_{ij}$ corresponding to the cones $\sigma_i, \sigma_{ij}$. 
For $u>1$, the  task is to  show that if such a collection of characters
corresponds to a Cartier divisor then there is an $m \in M$ with  $m_i = m_{ij} = m$ for all $i, j$.

We can assume  $m_n = 0$ without loss of generality, such that the corresponding toric Cartier
divisor is of the form $-c D_{f_n} - \sum_{i = 1}^{n - 1} c_i D_{g_i} - c_h D_h$. Then it follows that $m_i(f_n) = c$ for
every $1 \leq i < n$. Moreover, for $1 \leq j \neq i < n$ we have $m_i(f_j) = 0$. Altogether we have a complete set of
linearly independent conditions which determine $m_i$ and we obtain $c_i = m_i(g_i) = c$ for every $1 \leq i < n$.

Next we consider any cone $\sigma_{ij}$, with $1 \leq i \neq j < n$. Then we have $m_{ij}(g_i) = m_{ij}(g_j) =
m_{ij}(f_n)  = c$, and $m_{ij}(f_k) = 0$ for every $k \neq i, j, n$. It follows that $m_{ij}(g_i) = m_{ij}(h - f_i)
= m_{ij}(g_j) = m_{ij}(h - f_j) = c$, hence $m_{ij}(f_i) = m_{ij}(f_j) = m_{ij}(h) - c$ and therefore $m_{ij}(f_n) =
-2m_{ij}(f_i) = c$, so we get $c_h = m_{ij}(h) = c / 2$. At this point, we have shown that $\Pic(Y_u)$ is exhausted
by the parameter $c$ and therefore has rank at most one.

Now we consider $\sigma_{in}$ for any $1 \leq i < n$. Again, we have $m_{in}(g_i) = c$, but $m_{in}(g_n) = m_{in}(f_k)
= 0$ for $1 \leq k \neq i < n$. With $m_{in}(g_i) = m_{in}(h) - m_{in}(f_i) = c/2 - m_{in}(f_i) = c$ it follows
$m_{in}(f_i) = -c/2$. Then $m_{in}(g_n) = uc/2 - m_{in}(f_n) = uc/2 + m_{in}(f_i) = (u-1)c/2$.
By our original assumption, we had $m_n(g_n) = m_{in}(g_n) = 0$, so for $u > 1$ we necessarily have $c = 0$
and hence a toric Cartier divisor is rationally equivalent to zero.

For the case $u = 1$, it is straightforward to check that for $c > 0$ the corresponding characters $m_i, m_{ij}$
constitute a strictly convex piecewise linear function on $\Delta_1$; we leave this as an exercise for the reader.
\qed

\begin{proposition}\label{Egyptianposition}
The ray generated by $e$ is in Egyptian position and the corresponding divisor $D_e$ on $Y_u$
is projective.
\end{proposition}

\proof
For any $1 \leq i \leq n$, the cone $\sigma_i'$ is generated by $f_j, j \neq k$ and $g_i$ and hence, by construction,
$\sigma_i$ is a pyramidal extension of $\sigma_i'$ by $e$. So, $D_e$ is Egyptian.
The star $\Star(\RR_{\geq 0} e)$ consists the maximal cones $\sigma_1, \dots, \sigma_n$ and, again by construction,
the fan in $N_\RR / \RR e$ generated by images of the $\sigma_i$ under the projection map is the fan associated to
$\mathbb{P}^{n - 1}$ and therefore we have $D_e \simeq \mathbb{P}^{n - 1}$.
\qed

\medskip
Now, by putting together Propositions \ref{trivialpic} and \ref{Egyptianposition} and Theorem \ref{manybundles},
we obtain:

\begin{theorem}
For all $n\geq 3$ and $u > 1$, the toric variety $Y_u$ has no nontrivial invertible sheaves but 
admits locally free sheaves $\shE$ of rank $n=\dim(Y_u)$
with arbitrarily large top Chern class $c_n(\shE)$.
\end{theorem}

%===========================================================
\section{Projective divisors on threefolds}
\mylabel{Divisors on threefolds}

Let $k$ be a ground field.
Theorem \ref{infinitely classes} triggers the following question: Under what conditions
does a proper  scheme $X$ contain a divisor $D\subset X$ so that the proper scheme $D$ is projective?
As far as we see, the existence of such a projective divisor is open   for smooth proper 
threefolds that are non-projective. In this direction, we have   a partial result:

\begin{proposition}
Let $X$ be an integral, normal, proper threefold that is $\QQ$-factorial,
and $S\subset X$ be an irreducible closed subscheme of dimension $\dim(S)=2$.
Suppose there is a quasiprojective open subset $U\subset X$ containing all points
$x\in S$ of codimension $\dim(\O_{S,x})=1$. Then the proper scheme $S$ is projective.
\end{proposition}

\proof
First note that we may assume that the structure sheaf $\O_S$ contains no nontrivial torsion sections.
This follows inductively from the following observation: If $\shJ\subset\O_S$ is a quasicoherent ideal
sheaf with  $\dim(\shJ)\leq 1$, defining a closed subscheme $S'\subset S$,
 we have a short exact sequence of abelian sheaves
$$
1\lra 1+\shJ\lra\O_{S}^\times\lra\O_{S'}^\times\lra 1,
$$
In the long exact sequence
$$
\Pic(S)\lra\Pic(S')\lra H^2(S,1+\shJ),
$$
the term on the right vanishes because the sheaf $1+\shJ$ is supported on a closed subset of dimension $<2$. 
It follows from \cite{EGA II}, Proposition 4.5.13 that $S$ is projective provided that $S'$ is projective.

According to Chow's Lemma (in the refined form of \cite{Deligne 2010}, Corollary 1.4), there is a proper morphism $f:\tilde{X}\ra X$
with $\tilde{X}$ projective and $f^{-1}(U)\ra U$ an isomorphism. Clearly, we may also assume that $\tilde{X}$
is integral and normal. Let $R\subset \tilde{X}$ be the exceptional locus, which we regard as a reduced closed subscheme.
After replacing $\tilde{X}$ by a blowing-up with center $R$, we may assume that $R\subset\tilde{X}$ is 
a Cartier divisor.
Let $\tilde{S}\subset \tilde{X}$ be the strict transform of the surface $S$,
that is, the schematic closure of $f^{-1}(U\cap S)\subset\tilde{X}$. 

Choose an ample sheaf $\shL\in\Pic(\tilde{X})$. Replacing $\shL$ by a suitable multiple, we may
assume that $\shN=\shL(-R)$ is ample as well. 
Let $R=R_1\cup\ldots\cup R_t$ be the irreducible components. Since $\tilde{S}$ intersects $f^{-1}(U)$, the scheme $\tilde{S}$
contains none of the $R_i$, hence there are closed points $r_i\in R_i\smallsetminus\tilde{S}$.
Let $\shI\subset\O_{\tilde{X}}$ be the quasicoherent ideal corresponding to the closed subscheme $\tilde{S}\cup A\subset\tilde{X}$, 
where $A=\left\{r_1,\ldots,r_t\right\}$, and the union is disjoint.
Choose some $n_0$ so that for all $n\geq n_0$, $H^1(\tilde{X},\shN^{\otimes n}\otimes\shI)=0$.
Next, choose some $n\geq n_0$ so that there is a  regular section $s'\in H^0(\tilde{S},\shN^{\otimes n}_{\tilde{S}})$
whose zero set $(s'=0)\subset\tilde{S}$ is irreducible.
The short exact sequence
$$
0\lra\shI\otimes\shN^{\otimes n}\lra \shN^{\otimes n}\lra \shN^{\otimes n}|_{\tilde{S}\cup A}\lra 0
$$
yields an exact sequence
$$
H^0(\tilde{X},\shN^{\otimes n})\lra H^0(\tilde{S}\cup A,\shN^{\otimes n}|_{\tilde{S}\cup A})\lra  H^1(\tilde{X},\shN^{\otimes n}\otimes\shI),
$$
where the term on the right vanishes, and the term in the middle is a sum corresponding to the 
disjoint union $\tilde{S}\cup A$.
 Therefore we may extend $s'$ to a section $s$ over $\tilde{X}$ that is nonzero at each generic point of 
$R$. It thus  defines an ample Cartier divisor
$\tilde{H}\subset\tilde{X}$ whose intersection with $\tilde{S}$ is irreducible, and that contains no irreducible component
of the exceptional divisor $R$. We now consider its image $H=f(\tilde{H})$, which is a closed subscheme of $X$.

\medskip
{\bf Claim 1:} Each irreducible component of $H$ is of codimension one in $X$.
Indeed: The irreducible components   $\tilde{H}_i\subset\tilde{H}$ are of codimension one in $\tilde{X}$,
and their generic points lie in $f^{-1}(U)=U$. If follows that their images $H_i=f(\tilde{H}_i)$ have codimension one.

\medskip
{\bf Claim 2:} The scheme $S\cap H$ is irreducible and 1-dimensional.
Clearly, we have a   union
$$
S\cap H = f(\tilde{S}\cap\tilde{H}) \cup \bigcup_{s\in S} f((f^{-1}(s)\smallsetminus \tilde{S})\cap \tilde{H}).
$$
By construction, $\tilde{S}\cap\tilde{H}$ is irreducible, and so is its image $f(\tilde{S}\cap\tilde{H})$.
The sets on the right $(f^{-1}(s)\smallsetminus \tilde{S})\cap \tilde{H}$ can be nonempty only if $s\in S$ 
is a critical point for $f:\tilde{X}\ra X$, that is, in
the image of the exceptional set $R\subset\tilde{X} $, thus contained in $S\smallsetminus U$, which is finite.
The upshot is that $S\cap H$ is a disjoint union of the irreducible closed subset
$f(\tilde{S}\cap\tilde{H})$ and finitely many closed points $x_1,\ldots,x_r\in S$.
Now we use the assumption that $X$ is $\QQ$-factorial: The closed subset $H\subset X$ is the support of some Cartier divisor,
so $S\cap H\subset S$ is the support of some Cartier divisor. In turn,  each irreducible component of $S\cap H$ is purely of codimension one.
It follows that $S\cap H = f(\tilde{S}\cap\tilde{H})$, and this must be 1-dimensional.

\medskip
{\bf Claim 3:} The scheme $S\smallsetminus (S\cap H)$ is affine.
To see this, note that 
$$
H=f(\tilde{H})=f(\tilde{H}\cup R)
$$
This is  because $\tilde{H}$ intersects each curve on $\tilde{X}$, in particular those mapping to points in $X$,
and the latter cover $R$. Since $\shN^{\otimes n}(nR)=\shL^{\otimes n}$ is ample, the effective Cartier divisor
$\tilde{H}\cup  R$ is ample, hence its complement $\tilde{X}\smallsetminus(\tilde{H} \cup R)$ is affine. 
Clearly, $\tilde{H}\cup R$ is saturated with respect to the map $f:\tilde{X}\ra X$, thus $\tilde{H}\cup R=f^{-1}f(\tilde{H}\cup R)=f^{-1}(H)$.
By construction,
$$
X\smallsetminus f(R) \supset X\smallsetminus f(\tilde{H}\cup R) = X\smallsetminus H.
$$
We conclude that $f:\tilde{X}\ra X$ is an isomorphism over $X\smallsetminus H$,
and that
$$
\tilde{X}\smallsetminus f^{-1}(H) \simeq  \tilde{X}\smallsetminus (\tilde{H}\cup R)\simeq X\smallsetminus H
$$
is affine.

\medskip
{\bf Claim 4:} The scheme $S$ is projective.
Since $X$ is $\QQ$-factorial, we may   endow the closed subset $H\subset X$ whose irreducible components are 1-codimensional
with a suitable scheme structure so that it becomes an effective Cartier divisor.
Then $S\cap H$ is an effective Cartier divisor, which is moreover irreducible, and has affine complement.
According to Goodman's Theorem (\cite{Goodman 1969}, Theorem 2  on page 168), there 
is an ample divisor on $S$ supported by $S\cap H$, in particular $S$ is projective. Goodman formulated his result under the
assumption that the local rings $\O_{S,s}$ are factorial for all $s\in S$,
but the proof goes through with only minor modification under 
our assumption of $\QQ$-factoriality. Compare also \cite{Hartshorne 1970}, Chapter II, \S 4, Theorem 4.2
for a nice exposition of Goodman's arguments.
\qed

\medskip
The following observation emphasizes that the existence of large quasiprojective open subsets
is a delicate condition:

\begin{proposition}
Let $X$ be an integral, normal, proper $n$-fold that is $\QQ$-factorial but does not admit  an ample invertible sheaf.
Then there is no quasiprojective open subset $U\subset X$ containing all points $x\in X$ of codimension $\dim(\O_{X,x})=n-1$.
\end{proposition}

\proof
Suppose there would be such an open subset $U\subset X$. Choose a very ample divisor $H_U\subset U$,
and let $H\subset X$ be its closure. Since $X$ is $\QQ$-factorial, we may assume that $H\subset X$
is Cartier. Let $\shL=\O_X(H)$ be the corresponding invertible sheaf. Obviously,
its base locus is contained in $A=X\smallsetminus U$, which is finite.
By the Zariski--Fujita Theorem \cite{Fujita 1983}, we may replace  $\shL$ by some tensor power
and assume that $\shL$ is globally generated. Let $f:X\ra\PP^m$ be the morphism
coming from the linear system $H^0(X,\shL)$, with $m+1=h^0(\shL)$.
Clearly, the morphism $f$ is injective on $U$, thus has finite fibers.
Then it follows from \cite{EGA IIIa}, Proposition 2.6.2 that the scheme $X$ is projective, contradiction.
\qed

%===========================================================

\end{document}